\documentclass[11pt,reqno]{article}

\usepackage[top=2.5cm, bottom=2.5cm, left=2.5cm, right=2.5cm]{geometry}	

\RequirePackage{amsthm,amsmath,amsfonts,amssymb}
\RequirePackage[numbers]{natbib}
\RequirePackage[colorlinks,citecolor=blue,urlcolor=blue]{hyperref}
\RequirePackage{graphicx}

\newtheorem{theorem}{Theorem}
\newtheorem{corollary}{Corollary}

\newtheorem{lemma}{Lemma}

\theoremstyle{remark}

\newtheorem{example}{Example}
\newtheorem{remark}{Remark}

\usepackage{bm} 
\usepackage{color}
\usepackage[dvipsnames]{xcolor}
\usepackage{comment}

\newcommand*\pFq[2]{{}_{#1}F_{#2}}

\renewcommand{\cal}[1]{\mathcal{#1}}

\renewcommand{\r}{\mathbb{R}}

\newcommand{\mmag}[1]{\left|#1\right|}

\newcommand{\Ebb}[1]{\mathbb{E}\left[#1\right]}
\newcommand{\Pbb}[1]{\mathbb{P}\left(#1\right)}

\hypersetup{
  colorlinks   = true, 
  urlcolor     = {green!50!black}, 
  linkcolor    = {green!50!black}, 
  citecolor   = {green!50!black} 
}

\title{Product-form estimators: exploiting independence to scale up Monte Carlo\thanks{JK and AMJ acknowledge support from
the EPSRC (grant \# EP/T004134/1) and the Lloyd's Register Foundation Programme on Data-Centric Engineering at the Alan Turing Institute. FRC acknowledges support from the EPSRC and the MRC OXWASP Centre for Doctoral Training (grant \# EP/L016710/1). FRC and AMJ acknowledge further support  from the EPSRC (grant \#  EP/R034710/1).}}

\author{
  Juan Kuntz\thanks{Department of Statistics, University of Warwick, Coventry, CV4 7AL, UK. Email: $\{$juan.kuntz-nussio, francesca.crucinio, a.m.johansen$\}$@warwick.ac.uk} \thanks{Alan Turing Institute, 96 Euston Road, London, NW1 2DB, UK.}
  \and
  Francesca R. Crucinio\footnotemark[2]
  \and
  Adam M. Johansen\footnotemark[2] \footnotemark[3]
}

\begin{document}

\maketitle

\begin{abstract}

We introduce a class of Monte Carlo estimators that aim to overcome the rapid growth of variance with dimension often observed for standard estimators by exploiting the target's independence structure. 
We identify the most basic incarnations of these estimators with a class of generalized U-statistics, and thus establish their unbiasedness, consistency, and asymptotic normality. 
Moreover, we show that they obtain the minimum possible variance amongst a broad class of estimators; and we investigate their computational cost and delineate the settings in which they are most efficient.  We exemplify the merger of these estimators with other well-known Monte Carlo estimators so as to better adapt the latter to the target's independence structure and improve their performance. We do this via three simple mergers: one with importance sampling, another with importance sampling squared, and a final one with pseudo-marginal Metropolis-Hasting. In all cases, we show that the resulting estimators are well-founded and achieve lower variances than their standard counterparts. Lastly, we illustrate the various variance reductions through several examples.

\end{abstract}

\section{Introduction}\label{sec:intro}
Monte Carlo methods are sometimes said to overcome the curse of dimensionality because, regardless of the target's dimension, their rates of convergence are square root in the number of samples drawn. In practice, however, one encounters several problems when computing high-dimensional integrals using Monte Carlo, prominent among which is the issue that the constants present in the convergence rates typically grow rapidly with the target's dimension. Hence, even if we are able to draw independent samples from a high-dimensional target, the number of  samples necessary to obtain estimates of a satisfactory accuracy is often prohibitively large~\cite{Silverman1986,Snyder2008,Bengtsson2008,Agapiou2017}.  However, many of these targets possess strong independence structures (e.g.\ see \cite{Gelman2006,Gelman2006a,Koller2009,hoffman2013stochastic,Blei2003} and the many references therein). In this paper, we investigate whether the rapid growth of these constants  can be mitigated by exploiting these structures.

Variants of the following toy example are sometimes given to illustrate the issue~(e.g.\ \cite[p.95]{chopin2020}).
Let $\mu$ be a $K$-dimensional  isotropic Gaussian distribution   with unit means and variances, and consider the basic Monte Carlo estimator for the mean ($\mu(\varphi)=1$) of the product ($\varphi(x):=x_1x_2\ldots x_K$) of its components ($x_1,\ldots, x_K$):
\begin{equation}\label{eq:fney8safbay8fa}\mu^N(\varphi):= \frac{1}{N}\sum_{n=1}^N\varphi(X_1^n,\dots ,X_K^n)=\frac{1}{N}\sum_{n=1}^NX_1^n\cdots X_K^n,\end{equation}
where $(X_1^n,\dots,X_K^n)_{n=1}^N$ denote i.i.d.\ samples drawn from $\mu$. Because the estimator's asymptotic variance equals $2^K-1$, the number of samples required to obtain a reasonable estimate of $\mu(\varphi)$   grows exponentially with the target's dimension. Hence, it is impractical to use $\mu^N(\varphi)$ if $K$ is even modestly large. For instance, if $K=20$, we would require $\approx 10^{10}$ samples to obtain an estimate with standard deviation of $0.01=1\%\mu(\varphi)$, reaching the limits of most present-day personal computers, and if $K=30$, we would require $ \approx 10^{13}$ samples, exceeding these limits. 

There is, however, a trivial way of overcoming  the issue for the above example that does not require any knowledge about $\mu$ beyond the fact that it is product-form. 
Because $\mu$ is the product $\mu_1\times\dots\times\mu_K$ of $K$ univariate unit-mean-and-variance Gaussian distributions $\mu_1,$ $\dots,\mu_K$ and $\varphi$ is the product $\varphi_1\cdots\varphi_K$ of $K$ univariate functions $\varphi_1(x_1)=x_1,\dots,\varphi_K(x_K)=x_K$,  we can express $\mu(\varphi)$  as the product $\mu_1(\varphi_1)  \cdots \mu_K(\varphi_K)$ of the corresponding $K$ univariate means $\mu_1(\varphi_1),\dots,\mu_K(\varphi_K)$. As we will see in Section~\ref{sec:kproducts}, estimating each of these means separately and taking the resulting product, we obtain an estimator for $\mu(\varphi)$ whose asymptotic variance is $K$:
\begin{align}\label{eq:intrvar}\mu^N_\times(\varphi):=& \frac{1}{N^K}\sum_{n_1=1}^N\dots\sum_{n_K=1}^N\varphi(X_1^{n_1},\dots,X_K^{n_K})=
  \left(\frac{1}{N}\sum_{n_1=1}^NX_1^{n_1}\right) \cdots\relax  \left(\frac{1}{N}\sum_{n_K=1}^N X_K^{n_K}\right).\end{align}
Consequently, the number of samples necessary for $\mu^N_\times(\varphi)$ to yield a reasonable estimate of $\mu(\varphi)$ only grows  linearly with the dimension, allowing us to practically deal with $K$s in the millions.

The middle term of~\eqref{eq:intrvar} makes sense regardless of whether $\varphi$ is the product of univariate test functions. It defines a type  of (unbiased, consistent, and asymptotically normal) Monte Carlo estimators for general $\varphi$ and product-form $\mu$ which   we refer to as \emph{product-form estimators}. Their salient feature  is that they achieve lower variances than the standard estimator~\eqref{eq:fney8safbay8fa} given the same number of samples from the target. The reason behind the variance reduction is simple:  if $(X^n_1)_{n=1}^N$,$\dots$, $(X^n_K)_{n=1}^N$ are independent sequences of samples respectively drawn  from $\mu_1,\dots,\mu_K$, then every `permutation' of these samples has  law $\mu$, that is,

\begin{equation}\label{eq:permute}(X^{n_1}_1,\dots,X^{n_K}_K)\sim\mu \quad\forall n_1,\dots,n_K\leq N.\end{equation}
Hence, $\mu^N_\times(\varphi)$ in~\eqref{eq:intrvar} averages over $N^K$ tuples with law $\mu$ while its conventional counterpart~\eqref{eq:fney8safbay8fa} only averages over $N$ such tuples. This increase in tuple number leads to a decrease in estimator variance and we say that the product-form estimator is more \emph{statistically efficient} than the standard one. Moreover, obtaining these $N^K$ tuples does not require drawing any further samples from $\mu$ and, in this sense, product-form estimators make the most out of every sample available (indeed, we will show in Theorem~\ref{thrm:mvu} that they are minimum variance unbiased estimators, or MVUEs, for product-form targets). However, in contrast with the tuples in~\eqref{eq:fney8safbay8fa}, those in~\eqref{eq:intrvar} are not independent (the same components are repeated across several tuples). For this reason, product-form estimators achieve the same $\cal{O}(N^{-1/2})$ rate of convergence  that the standard ones do and  the variance reduction materializes only in lower proportionality constants (i.e.\ $\lim_{N\to\infty}\text{Var}(\mu^N_\times(\varphi))/\text{Var}(\mu^N(\varphi))=C$ for some constant $C\leq1$).

The space complexity of product-form estimators scales linearly with dimension: to utilize all $N^K$ permuted tuples in~\eqref{eq:intrvar} we need only store $KN$ numbers ($X^1_1,\ldots, X^N_1; \ldots; X^1_K, \ldots, X^N_K$). However, in the absence of any sort of special structure in the test functions, their time complexity  scales exponentially with dimension: brute-force computation of  the sum in~\eqref{eq:intrvar} requires\footnote{\textbf{On our  $\cal{O}$ notation:} The exact dependence on dimension of the estimators' evaluation costs depends on that of the test function $\varphi$. Hence, when discussing a generic  $\varphi$, we say that the estimator's evaluation cost is $\cal{O}(N^d)$ for some $d$ to mean that it is $\cal{O}(f(K)N^d)$ for some unspecified factor $f(K)$ that accounts for $\varphi$'s evaluation cost. When discussing classes of $\varphi$ for which this factor is clear, we specify it. For example, we say that evaluation cost of the rightmost term in~\eqref{eq:intrvar} is $\cal{O}(KN)$ rather than $\cal{O}(N)$.} $\mathcal{O}(N^K)$ operations. Consequently, the use of product-form estimators for general $\varphi$ proves to be a balancing act in which one must weigh the cost of acquiring new samples from $\mu$ (be it a computational one if the samples are obtained from simulations, or a real-life one if they are obtained from experiments) against the extra overhead required to evaluate these estimators, and it is limited to $K$s no greater than ten.

If, however, the test function $\varphi$ also possesses some `product structure', then $\mu_\times^N(\varphi)$ can often be evaluated in far fewer than $\cal{O}(N^K)$ operations. The most extreme examples of such $\varphi$ are functions that factorize fully and sums thereof (which we refer to as `sums of products' or `SOPs'), for which the evaluation cost is easily lowered to just $\mathcal{O}(KN)$. For instance, in the case of the toy Gaussian example above, we can evaluate the product-form estimator in $\cal{O}(KN)$ operations by expressing it as the product of the component-wise sample averages and computing each average separately (i.e.\ using the final expression in~\eqref{eq:intrvar}). This cheaper approach just amounts to a dimensionality reduction technique: we re-write a high-dimensional integral as a polynomial of low-dimension integrals, estimate each of low-dimension integral separately, and plug the estimates back into the polynomial to obtain an estimate of the original integral. More generally, if the test function can be expressed as a sum of partially-factorized functions, it is often possible to lower the cost to $\cal{O}(N^d)$ where $d< K$ depends on the amount of factorization, and taking this approach also amounts to a type of dimensionality reduction (this time featuring nested integrals).

This paper has two goals. First, to provide a comprehensive theoretical characterization of product-form estimators. Second, to illustrate their use for non-product targets when combined with, or embedded within, other more sophisticated Monte Carlo methodology. It is in these settings, where  product-form estimators are deployed to tackle the aspects of the problem exhibiting product structure or conditional independences, that we believe these estimators find their greatest use.  To avoid unnecessary technical distractions, and in the interest of accessibility, we achieve the second goal using simple examples. While we anticipate that the most useful such combinations or embeddings will not to be so simple, we believe that the underlying ideas and guiding principles will be the same.

\paragraph*{Relation to the literature.}

$\enskip$ In their basic form, product-form estimators~\eqref{eq:intrvar} are a subclass of generalized U-statistics (see \cite{Lee1990,korolyuk2013theory} for comprehensive surveys): multisample U-statistics with `kernels' $\varphi$ that take as arguments a \emph{single} sample per distribution for several distributions $(K>1)$. Even though product-form estimators are unnatural examples of U-statistics because the original unisample U-statistics~\cite{hoeffding1948} fundamentally involve symmetric kernels that take as arguments multiple samples from a single distribution ($K=1$), the methods used to study either of these overlap significantly. The arguments required in the basic product-form case are simpler than those necessary for the most general case (multiple samples from multiple distributions) and, by focusing on the results that are of greatest interest from the Monte Carlo perspective, we are able to present readily accessible, intuitive, and compact proofs for the theoretical properties of~\eqref{eq:intrvar}. This said, whenever a result given here can be extracted from the U-statistics literature, we provide explicit references.

While U-statistics have been extensively studied since Hoeffding's seminal work~\cite{hoeffding1948} and are commonly employed in a variety of statistical tests (e.g.\ independence tests \cite{Hoeffding1948a}, two-sample tests~\cite{Gretton2012}, goodness-of-fit tests~\cite{Liu2016}, and more \cite{Lee1990,Kowalski2007}) and learning tasks (e.g.\ regression~\cite{Kowalski2007}, classification~\cite{Clemencon2008}, clustering~\cite{clemencon2011u}, and more~\cite{Clemencon2008,Clemencon2016}) where they arise as natural estimators, their use in Monte Carlo seems underexplored. Exceptions include \cite{owen2009} which cleverly applies unisample U-statistics  to  make the best possible use of a collection of genuine (and hence expensive to obtain and store) uniform random variables and  \cite{hall1987estimation} that uses them to obtain improved estimates for the integrated squared derivatives of a density.
 
Product-form estimators themselves can be found peppered throughout the Monte Carlo literature, with one exception (see below), always unnamed and specialized to particular contexts.  First off, in the simplest setting of integrating fully-factorized functions with respect to product-form measures, it is of course well-known   that better performance is obtained by separately approximating the marginal integrals and taking their product (although, we have yet to locate full variance expressions quantifying quite how much better,  even for this near-trivial case). Beyond this case, product-form estimators are found not in isolation but combined with other Monte Carlo methodology: \cite{Tran2013} embeds them within therein-defined Importance Sampling$^2$ (IS$^2$) to efficiently infer parameters of structured latent variable models, \cite{Schmon2021} employs them within pseudo-marginal MCMC to estimate intractable acceptance probabilities for similar models, \cite{lindsten2017divide,Kuntz2021} study their use within Sequential Monte Carlo (SMC), and~\cite{Aitchison2019} builds on them to obtain Tensor Monte Carlo (TMC), an extension of importance weighted variational autoencoders~\cite{Burda2016}. The latter article is the aforementioned exception: its author defines the estimators in general and refers to them as `TMC estimators', but does not study them theoretically. To the best of our knowledge, there has been no previous systematic exploration of the estimators~\eqref{eq:intrvar}, their theoretical properties, and uses, a gap we intend to fill here. 
Furthermore, while in simple situations with fully, or almost-fully, factorized test functions (e.g.\ those in \cite{Tran2013,Schmon2021}) it might be clear to most practitioners that employing a product-form estimator is the right thing to do, it may not be quite so immediately obvious how much of a difference this can make and that, in rather precise ways (cf.\ Theorems~\ref{thrm:mvu}~and~\ref{thrm:mvuppf}), judiciously using product-form estimators is the best thing one can do within Monte Carlo when tackling models with known independence structure but unknown conditional distributions (a common situation in practice). We aim to underscore these points through our analysis and examples.
 
Lastly, we remark that product-form estimators are reminiscent of classical product cubature rules~\cite{Stroud1971}. These are obtained by taking products of quadrature rules and, consequently, require computing sums over $N^K$ points much like for product-form estimators (except for fully, or partially, factorized test functions $\varphi$ where the cost can be similarly lowered, e.g.\ \cite[p.24]{Stroud1971}). In fact, the high computational cost incurred by these rules for general $\varphi$ partly motivated the development of more modern numerical integration techniques such as Quasi Monte Carlo~\cite{Dick2013}, spare grid methods~\cite{gerstner1998numerical,Gerstner2003}, and, of course, Monte Carlo itself. That said, we believe that these rules can be used to great effect if one is strategic in their application and the advent of the more modern methods has created many opportunities for such applications, something we aim to exemplify here using their Monte Carlo analogues: product-form estimators. 

\paragraph*{Paper structure.} $\enskip$
This paper is divided into two main parts (Sections~\ref{sec:motivation}~and~\ref{sec:extend}), each corresponding to one of our two aims, and a discussion of our results, future research directions, and potential applications (Section~\ref{sec:discussion}). 

Section~\ref{sec:motivation} studies product-form estimators and their theoretical properties. We show that the estimators are strongly consistent, unbiased, and asymptotically normal, and we give expressions for their finite sample and asymptotic variances (Section~\ref{sec:kproducts}). We argue that they are more  statistically efficient than their conventional counterparts in the sense that they achieve lower variances given the same number of samples (Section~\ref{sec:statseff}). Lastly, we consider their computational cost (Section~\ref{sec:compeff}) and explore the circumstances in which they prove most computationally efficient (Section~\ref{sec:lincost}). 

Section~\ref{sec:extend} gives simple examples illustrating how one may embed product-form estimators within standard Monte Carlo methodology and extend their use beyond product-form targets. In particular, we combine them with importance sampling and obtain estimators applicable to targets that are absolutely continuous with respect to product-form distributions (Section~\ref{sec:is}) and partially-factorized distributions (Section~\ref{sec:ppf}), and we consider their use within pseudo-marginal MCMC (Section~\ref{sec:pMH}). We then examine the numerical performance of these extensions on a simple hierarchical model (Section~\ref{sec:numex}).

The paper has six appendices. The first five contain proofs: Appendix~\ref{app:classicalproof} those for the basic properties of product-form estimators, Appendix~\ref{app:mvu} that for their MVUE property, Appendix~\ref{app:classicalK2} those for the basic properties of the `partially product-form' estimators introduced in Section~\ref{sec:ppf}, Appendix~\ref{app:ppfmvu} that for the latter's MVUE property, and Appendix~\ref{app:pmh} that for the statistical efficiency of the product-form pseudo-marginal MCMC estimators (vis-\`a-vis their non-product counterparts) in Section~\ref{sec:pMH}.
Appendix~\ref{sec:mop} contains an additional, simple extension of product-form estimators (to targets that are mixture of product-form distributions), omitted from the main text in the interest of brevity.

\section{Product-form estimators}\label{sec:motivation}

Consider the basic Monte Carlo problem: given  a probability distribution $\mu$ on a measurable space $(S,\cal{S})$ and a function $\varphi$ belonging to the space $L^2_\mu$ of square $\mu$-integrable real-valued functions on $S$, estimate the average
$$\mu(\varphi):=\int\varphi(x)\mu(dx).$$
Throughout this section, we focus on the question `by exploiting the product-form structure of a target $\mu$, can we design estimators of  $\mu(\varphi)$ that are more efficient than the usual ones?'. By product-form, we mean that $\mu$ is the product of $K>1$ distributions $\mu_1,\dots,\mu_K$ on measurable spaces $(S_1,\cal{S}_1),\dots,(S_K,\cal{S}_K)$ satisfying $S=S_1\times\dots\times S_K$ and $\cal{S}=\cal{S}_1\times\dots\times \cal{S}_K$, where the latter denotes the product sigma-algebra. Furthermore, if $A$ is a non-empty subset of $[K]:=\{1,\dots,K\}$, then we use
$\mu_A:=\prod_{k\in A}\mu_k$
to denote the product of the $\mu_k$s indexed by $k$s in $A$ and  $\mu_A(\varphi)$ to denote the measurable function on $\prod_{k\not\in A}S_k$ obtained by integrating the arguments of $\varphi$ indexed by $k$s in $A$ with respect to $\mu_A$:
$$\mu_A(\varphi)(x_{A^c}):=\int\varphi(x_A,x_{A^c})\mu_A(dx_A)\quad\forall x_{A^c}\in \prod_{k \in A^c}S_k,$$
under the assumption that this integral is well-defined for all $x_{A^c}$ in $\prod_{k\not\in A}S_k$, where $A^c:=[K]\backslash A$ denotes $A$'s complement. If $A$ is empty, we set $\mu_A(\varphi):=\varphi$.
\subsection{Theoretical characterisation}\label{sec:kproducts} \sloppy Suppose that we have at our disposal $N$ i.i.d.\ samples $X^1,\dots,X^N$ drawn from $\mu$. We can view these samples as $N$ tuples 
$$(X^1_1,\dots,X^1_K),\enskip \dots,\enskip (X^N_1,\dots,X^N_K)$$
of i.i.d.\ samples $X^1_1,\dots, X^N_1$, $\dots$, $X^1_K,\dots, X^N_K$ independently and respectively drawn from $\mu_1,\dots,\mu_K$. As we will see in Section~\ref{sec:statseff},  the \emph{product-form estimator},
\begin{equation}
\label{eq:prodestK}\mu_\times^N(\varphi):=\frac{1}{N^K}\sum_{\bm{n}\in[N]^K}\varphi(X^{\bm{n}})
\end{equation}
where $X^{\bm{n}}$ with $\bm{n}=(n_1,\dots,n_K)$ denotes the `permuted' tuple $(X^{n_1}_1,\dots,X^{n_K}_K)$ (i.e. a tuple obtained as one of the $N^K$ component-wise permutations of the original samples),  is a better estimator for $\mu(\varphi)$ than the conventional choice using the same samples,
\begin{equation}\label{eq:basicest}\mu^N(\varphi):=\frac{1}{N}\sum_{n=1}^N\varphi(X^n),\end{equation}
regardless of whether the test function $\varphi$ possesses any sort of product structure. The reason behind this is as follows: while the conventional estimator directly approximates the target with the samples' empirical distribution,
\begin{equation}\label{eq:upn}\mu\approx\frac{1}{N}\sum_{n=1}^N\delta_{X^n}=:\mu^N,\end{equation}
the product-form estimator instead first approximates the marginals $\mu_1,\dots,\mu_K$ of the target with the corresponding component-wise empirical distributions,
$$\mu^N_1:=\frac{1}{N}\sum_{n=1}^N\delta_{X^n_1},\enskip\dots,\enskip\mu^N_K:=\frac{1}{N}\sum_{n=1}^N\delta_{X^n_K},$$
and then takes the product of these to obtain an approximation of $\mu$,
\begin{align}\mu\approx&\prod_{k=1}^K\left(\frac{1}{N}\sum_{n=1}^N\delta_{X^n_k}\right)=\frac{1}{N^K}\sum_{\bm{n}\in[N]^K}\delta_{X^{\bm{n}}} =:\mu^N_\times.\label{eq:upnx}\end{align}
The built-in product structure in $\mu^N_\times$ makes it a better suited approximation to the product-form target $\mu$ than the non-product $\mu^N$. Before pursuing this further, we take a moment to show that $\mu_\times^N(\varphi)$ is a well-founded estimator for $\mu(\varphi)$ and obtain expressions for its variance.
\begin{theorem}
\label{thrm:classicalK}
If $\varphi$ is $\mu$-integrable, then $\mu^N_\times(\varphi)$ in~\eqref{eq:prodestK} is unbiased:
$$\Ebb{\mu^N_\times(\varphi)}=\mu(\varphi)\quad\forall N>0.$$
If, furthermore, $\varphi$ belongs to $L^2_\mu$, then $\mu_{A^c}(\varphi)$ belongs to $L^2_{\mu_{A}}$ for all subsets $A$ of $[K]$. The estimator's variance is given by
\begin{align}
&\text{Var}(\mu^N_\times(\varphi))=\sum_{\emptyset\neq A\subseteq [K]}\frac{1}{N^{\mmag{A}}}\sum_{B\subseteq A}(-1)^{\mmag{A}-\mmag{B}}\sigma_{A,B}^2(\mu_{A^c}(\varphi)),\label{eq:vark}
\end{align}
for every $N>0$, where $\mmag{A}$ and $\mmag{B}$ denote the cardinalities of $A$ and $B$ and
\begin{equation}\label{eq:sigab}\sigma_{A,B}^2(\psi):=\mu_B([\mu_{A\backslash B}(\psi)-\mu_A(\psi)]^2)\end{equation}
for all  $\psi$ in $L^2_{\mu_A}$ and $B\subseteq A\subseteq[K]$. Furthermore, $\mu^N_\times(\varphi)$ is strongly consistent and asymptotically normal: 
\begin{align}\lim_{N\to\infty}\mu^N_\times(\varphi)=& \mu(\varphi)\enskip\textrm{almost surely,}\label{eq:lln}\\
N^{1/2}[\mu^N_\times(\varphi)-\mu(\varphi)]\Rightarrow& \cal{N}(0,\sigma^2_{\times}(\varphi))\enskip\text{as $N\to\infty$},\label{eq:clt}\end{align}
where $\Rightarrow$ denotes convergence in distribution and $\sigma_{\times}^2(\varphi):=\sum_{k=1}^K\sigma^2_k(\varphi)$ with 
$$\sigma^2_k(\varphi):=\mu_k([\mu_{\{k\}^c}(\varphi)-\mu(\varphi)]^2)\quad\forall k\in[N].$$
\end{theorem}
As mentioned in Section~\ref{sec:intro}, product-form estimators are special cases of multisample U-statistics and Theorem~\ref{thrm:classicalK} can be pieced together from  results in the U-statistics literature, e.g.\ see  \cite[p.35]{korolyuk2013theory} for the unbiasedness, \cite[p.38]{korolyuk2013theory}  for the variance expressions,  \cite[Theorem 3.2.1]{korolyuk2013theory} for the consistency (which also holds for $\mu$-integrable $\varphi$), \cite[Theorem 4.5.1]{korolyuk2013theory} for the asymptotic normality. To keep the paper self-contained we include a simple proof of Theorem~\ref{thrm:classicalK}, specially adapted for product-form estimators, in Appendix~\ref{app:classicalproof}. It has two key ingredients, the first being the following decomposition expressing  the `global approximation error' $\mu^N_\times-\mu$ as a sum of products of `marginal approximation errors' $\mu_1^N-\mu_1,\dots,\mu_K^N-\mu_K$:
\begin{align}
\mu^N_\times-\mu=&\prod_{k=1}^K\mu^N_{k}-\mu=\prod_{k=1}^K[(\mu^N_{k}-\mu_k)+\mu_k]-\mu=\sum_{\emptyset\neq A\subseteq [K]}\left(\prod_{k\in A}[\mu^N_{k}-\mu_k]\right)\times \mu_{A^c},\label{eq:decomp}
\end{align}
The other is the following expression  for the $L^2$ norm of a generic product of  marginal errors (see~\cite[p.152]{korolyuk2013theory} for its multisample U-statistics analogue). It tells us that the product of  $l$ of these errors has $\cal{O}(N^{-l/2})$ norm, as one would expect given that the errors are independent and that classical theory (e.g.\ \cite[p.168]{chopin2020}) tells us that the norm of each is $\cal{O}(N^{-1/2})$.
\begin{lemma}
\label{lem:iidl2}If $A$ is a non-empty subset of $[K]$, $ \psi$ belongs to $ L^2_{\mu_A}$, and $\sigma_{A,B}^2(\psi)$ is as in~\eqref{eq:sigab}, then
\begin{align*}
&\Ebb{\left[\left(\prod_{k\in A}[\mu_k^N-\mu_k]\right)(\psi)\right]^2}= \frac{1}{N^{\mmag{A}}}\sum_{B\subseteq A}(-1)^{\mmag{A}-\mmag{B}}\sigma_{A,B}^2(\psi)\quad\forall N>0.
\end{align*}
\end{lemma}
\begin{proof}This lemma follows from the equation
\begin{align}
&\left(\prod_{k\in A}[\mu^N_{k}-\mu_k]\right)(\psi)=\sum_{B\subseteq A}(-1)^{\mmag{A}-\mmag{B}}\mu_B^N(\mu_{A\backslash B}(\psi))=\mu_A^N\left(\sum_{B\subseteq A}(-1)^{\mmag{A}-\mmag{B}}\mu_{A\backslash B}(\psi)\right)=:\mu_A^N(\psi_A)\label{eq:proderr}
\end{align}
which, together with~\eqref{eq:decomp}, is known as \emph{Hoeffding's canonical decomposition} in the U-statistics literature~\cite[p.38]{korolyuk2013theory} and \emph{ANOVA-like} elsewhere~\cite{Efron1981} (similar decomposition are commonplace in the Quasi Monte Carlo literature, e.g.\ \cite[Appendix~A]{Owen2013}). 
See Appendix~\ref{app:classicalproof} for the details.
\end{proof}
\subsection{Statistical efficiency}\label{sec:statseff}
The product-form estimator $\mu^N_\times(\varphi)$ in~\eqref{eq:prodestK} yields the best  unbiased estimates of $\mu(\varphi)$ that can be achieved using only the knowledge that $\mu$ is product-form and $N$ i.i.d.\ samples drawn from $\mu$:
\begin{theorem}
\label{thrm:mvu}For any given measurable real-valued function $\varphi$ on $(S,\cal{S})$, $\mu_\times^N(\varphi)$ is an MVUE for $\mu(\varphi)$: if $f$ is a measurable real-valued function on $(S^N,\cal{S}^N)$ such that
$$\Ebb{f(X^1,\dots,X^N)}=\mu(\varphi)\enskip\text{with}\enskip X^1,\dots,X^N\sim\mu\enskip\text{i.i.d.}$$
for all product-form $\mu$ on $(S,\cal{S})$ satisfying $\mu(\mmag{\varphi})<\infty$, then 
$$\text{Var}(f(X^1,\dots,X^N))\geq\text{Var}(\mu^N_\times(\varphi)).$$
\end{theorem}
\begin{proof}See Appendix~\ref{app:mvu}.\end{proof}
While it is well-known that unisample U-statistics are  MVUEs (e.g.\ \cite{Clemencon2016}), we have been unable to locate an explicit proof that covers the general multisample case and, in particular, that of product-form estimators. Instead, we adapt the argument given in~\cite[Chap.\ 1]{Lee1990} (whose origins trace back to~\cite{Halmos1946}) for unisample U-statistics  and prove Theorem~\ref{thrm:mvu} in  Appendix~\ref{app:mvu}. 

Theorem~\ref{thrm:mvu} implies that product-form estimators achieve lower variances than their conventional counterparts:
\begin{corollary}
\label{cor:varbound}If $\varphi$ belongs to $L^2_\mu$ and $\sigma^2(\varphi):=\mu([\varphi-\mu(\varphi)]^2)$ denotes $\mu^N(\varphi)$'s asymptotic variance,
\begin{align*}
\text{Var}(\mu^N_\times(\varphi))\leq
\frac{\sigma^2(\varphi)}{N}=\text{Var}(\mu^N(\varphi))\quad\forall N>0,\qquad
\sigma_{\times}^2(\varphi)\leq \sigma^2(\varphi).
\end{align*}
\end{corollary}
\begin{proof}See Appendix~\ref{app:mvu}.\end{proof}
In other words, product-form estimators are more statistically efficient than their standard counterparts: using the same number of independent samples drawn from the target, $\mu^N_\times(\varphi)$ achieves lower variance than $\mu^N(\varphi)$. The reason behind this variance reduction was outlined in Section~\ref{sec:intro}: the product-form estimator uses the empirical distribution of the collection $(X^{\bm{n}})_{\bm{n}\in[N]^K}$ of permuted tuples as an approximation to $\mu$. Because $\mu$ is product-form, each of these permuted tuples is as much a sample drawn from $\mu$ as any of the original unpermuted tuples $(X^n)_{n=1}^N$. Hence, product-form estimators transform $N$ samples drawn from $\mu$ into $N^K$ samples and, consequently, lower the variance. However, the permuted tuples are not independent and we get a diminishing returns effect: the more permutations we make, the greater the correlations among them, and the less `new information' each new permutation affords us. For this reason, the estimator variance remains $\cal{O}(N^{-1})$, cf.\ \eqref{eq:vark}, instead of $\cal{O}(N^{-K})$ as would be the case for the standard estimator using $N^K$ independent samples. As we discuss in Section~\ref{sec:discussion}, there is also a pragmatic middle ground here: use $N<M<N^K$ permutations instead of all $N^K$ possible ones. In particular, by choosing these $M$ permutations to be as uncorrelated as possible (e.g. so that they have few overlapping entries), it might be possible to retain most of the variance reduction while avoiding the full $\cal{O}(N^K)$ cost (cf.\ \cite{kong2020design} and references therein for similar feats in the U-statistics literature).
 
Given that the variances of both estimators are (asymptotically) proportional to each other, we are now faced with the question `how large might the proportionality constant be?'. If the test function is linear or constant, e.g.\ $S_1=\dots=S_K=\r$ and
\begin{equation}
\label{eq:lintest}\varphi(x)=\sum_{k=1}^Kx_k,
\end{equation}
then the two estimators trivially coincide, no variance reduction is achieved, and the constant is one. However, these are the cases in which the standard estimator performs well (e.g.\ for \eqref{eq:lintest}, $\mu^N(\varphi)$'s variance breaks down into a sum of $K$ univariate integrals and, consequently, grows slowly with the dimension $K$).  However, if the test function includes dependencies between the components, then the proportionality constant can be arbitrarily large and the variance reduction unbounded as the following example illustrates.
\begin{example}\label{ex:iidsamplers}
If $K=2$,  $\mu_1=\mu_2=\cal{N}(0,1)$,  and  $\varphi(x):=1_{\{\min(x_1,x_2)\geq \alpha\}}(x)$, then
\begin{align*}
\mu(\varphi) = \mu(\varphi^2)= \left[1-\Phi\left( \alpha\right)\right]^2, \qquad \mu_1(\varphi)(x_2)=1_{\{x_2\geq\alpha\}}[1-\Phi(\alpha)],
\end{align*}
where $\Phi$ denotes the CDF of a standard normal, and similarly for $\mu_2(\varphi)(x_1)$. In addition,
\begin{align*}
\mu_1(\mu_2(\varphi)^2)=\mu_2(\mu_1(\varphi)^2)= \left[1-\Phi\left( \alpha\right)\right]^3.
\end{align*}
It then follows that
$$\frac{\sigma^2(\varphi)}{\sigma^2_{\times}(\varphi)}=\frac{ 2-\Phi\left( \alpha\right)}{2\left[1-\Phi\left( \alpha\right)\right]}\to\infty\quad\text{as}\quad\alpha\to\infty.$$

It is not difficult to glean some intuition as to why the product-form estimator yields far more accurate estimates than its standard counterpart for large $\alpha$. In these cases, unpermuted tuples with \emph{both} components greater than $\alpha$ are extremely rare (they occur with probability $[1-\Phi(\alpha)]^2$) and, until one arises, the standard estimator is stuck at zero (a relative error of $100\%$). On the other hand, for the product-form estimator to return a non-zero estimate, we only require  unpermuted tuples with a single component greater than $\alpha$, which are generated much more frequently (with probability $[1-\Phi(\alpha)]$).
\end{example}

Of particular interest is the case of high-dimensional targets (i.e.\ large $K$) for which obtaining accurate estimates of $\mu(\varphi)$ proves challenging.   Even though the exact manner in which the variance reduction achieved by the product-form estimator scales with dimension of course depends on the precise target and test function, it is straightforward to gain some insight by revisiting our starting example:
\begin{example}\label{ex:K}Suppose that
\begin{align*}
& S_1=\dots=S_K,\quad \cal{S}_1=\dots=\cal{S}_K, \quad \mu_1=\dots=\mu_K=\rho,\\
& \varphi=\prod_{k=1}^K\varphi_k,\quad \varphi_1=\dots=\varphi_K=\psi,
\end{align*}
for some univariate distribution $\rho$ and test function $\psi$ satisfying $\rho(\psi)\neq0$. In this case,
\begin{align}
\sigma^2(\varphi)&=\mu(\varphi^2)-\mu(\varphi)^2=\rho(\psi^2)^K-\rho(\psi)^{2K},\notag\\
\sigma^2_{\times}(\varphi)&=K\rho([\rho(\psi)^{K-1}[\psi-\rho(\psi)]]^2)=K\rho(\psi)^{2(K-1)}\rho([\psi-\rho(\psi)]^2)
=CV^2K\rho(\psi)^{2K},\label{eq:mfeu9anguiynguea}
\end{align}
where $CV:= {\sqrt{\rho([\psi-\rho(\psi)]^2)}}\big/{\rho(\psi)}$ denotes the coefficient of variation of $\psi$ w.r.t.\ $\rho$. Hence,
\begin{align}
&\frac{\sigma^2(\varphi)}{\sigma^2_{\times}(\varphi)}=\frac{(\rho(\psi^2)/\rho(\psi)^2)^K-1}{CV^2K}=\frac{(1+CV^2)^K-1}{CV^2K}=\frac{1}{K}\sum_{k=0}^{K-1}\binom{K}{k+1}CV^{2k},\label{eq:mfeu9anguiynguea2}
\end{align}
and we see that 
the reduction in variance grows exponentially with the dimension $K$.

At first glance,~\eqref{eq:mfeu9anguiynguea} might appear  to imply that the number of samples required for $\mu^N_\times(\varphi)$ to yield a reasonable estimate of $\mu(\varphi)$ grows exponentially with $K$ if $\mmag{\rho(\psi)}>1$. However, what we deem  a `reasonable estimate' should take into account the magnitude of the average $\mu(\varphi)$ we are estimating. In particular, it is natural to ask for the standard deviation of our estimates to be $\varepsilon\mmag{\mu(\varphi)}$ for some prescribed relative tolerance $\varepsilon>0$. In this case, we find that the number of samples required by the product-form estimator is approximately $$\sigma^2_{\times}(\varphi)/(\varepsilon^2\mu(\varphi)^2)=CV^2K\varepsilon^{-2}.$$ In the case of the conventional estimator $\mu^N(\varphi)$, the number required to achieve the same accuracy is instead $$\sigma^2(\varphi)/(\varepsilon^2\mu(\varphi)^2)=\varepsilon^{-2}((1+CV^2)^K-1).$$ That is, the number of samples necessary to obtain a reasonable estimate grows linearly with dimension for $\mu^N_\times(\varphi)$ and exponentially for $\mu^N(\varphi)$. 
\end{example}

\begin{figure*}[h!]
    \begin{center}
    \includegraphics[width=0.8\textwidth]{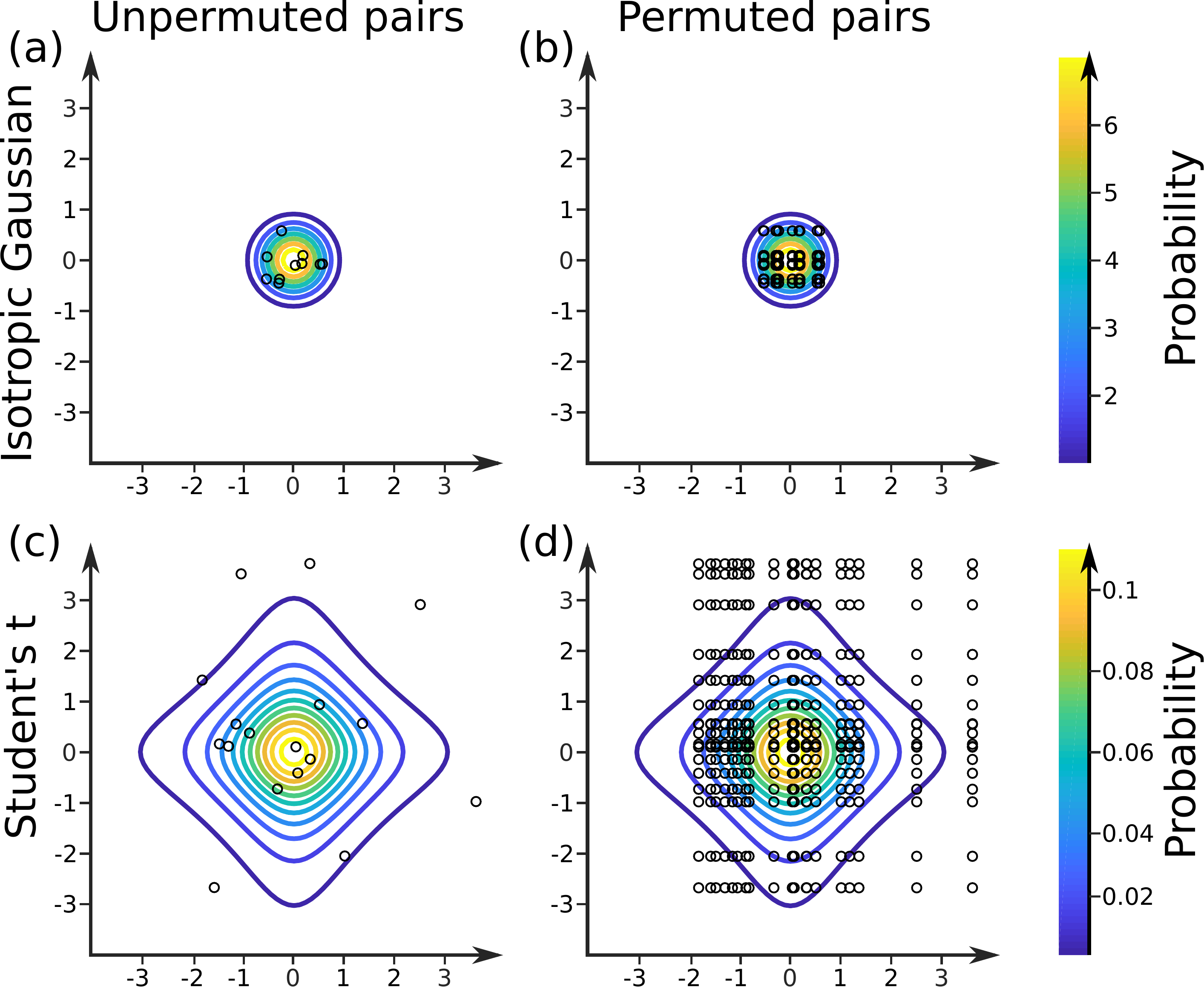}
    \vspace{-20pt}
    \end{center}
\caption{\textbf{Ensembles of unpermuted (a,c) and permuted (b,d) pairs for a peaked target (a,b) and heavy tailed one (c,d).} \textbf{(a)} $10$ pairs (dots) independently drawn from a two-dimensional isotropic Gaussian (contours) with mean zero and variance $0.1$. \textbf{(b)} The $100$ pairs (dots) obtained by permuting the pairs in (a). \textbf{(c)} $20$ pairs (dots) independently draw from the product of two student-t distributions (contours) with $1.5$ degrees of freedom. \textbf{(d)} The $400$ permuted pairs (dots) obtained by permuting the pairs in (c). }\label{fig:ldlike}
\end{figure*}

Notice that the univariate coefficient of variation $CV$ features heavily in Example~\ref{ex:K}'s analysis: the greater it is, the greater the variance reduction, and the difference gets amplified exponentially with the dimension $K$. This observation might be explained as follows: if $\mu$ is highly peaked (so that the coefficient is close to zero), then the unpermuted tuples are clumped together around the peak (Fig.\ \ref{fig:ldlike}a), permuting their entries only yields further tuples around the peak (Fig.\ \ref{fig:ldlike}b), and the empirical average changes little. If, on the other hand, $\mu$ is spread out (so that the coefficient is large), then the unpermuted pairs are scattered across the space (Fig.\ \ref{fig:ldlike}c), permuting their entries reveals unexplored regions of the space (Fig.\ \ref{fig:ldlike}d), and the estimates improve. Of course, how spread out the target is must be measured in terms of the test function and we end up with the coefficient of variation in~\eqref{eq:mfeu9anguiynguea2}.
\subsection{Computational efficiency}\label{sec:compeff}
As shown in the previous section, product-form estimators are always at least as statistically efficient  as their conventional counterparts: the variances of the former are bounded above by those of the latter. These gains in statistical efficiency come at a computational cost: even though both conventional and product-form estimators share the same $\cal{O}(N)$ memory needs, the latter requires evaluating the test function $N^K$ times, while the former requires only  $N$ evaluations. For this reason, the question of whether product-form estimators are more \emph{computationally efficient} than their conventional counterparts (i.e.\ achieve smaller errors given the same computational budget) is not as straightforward. In short, sometimes but not always. 

One way to answer  the computational efficiency question is to compare the cost incurred by each estimator in order to achieve a desired given variance $\sigma^2$. To do so, we approximate the variance of $\mu^N_\times(\varphi)$ with its asymptotic variance divided by the sample number (as justified by Theorem~\ref{thrm:classicalK}). The number of samples required for the variance to equal $\sigma^2$ is $N:=\sigma^2(\varphi)/\sigma^2$ for the conventional estimator and (approximately) $N_\times:=\sigma^2_{\times}(\varphi)/\sigma^2$ for the product-form one. The costs of evaluating the former with $N$ samples and the latter with $N_\times$ samples are $NC_{\varphi} +N C_X +N$  and $N^K_\times C_\varphi+N_\times C_X+N^K_\times$,
respectively, where $C_\varphi$ and $C_X$ are the costs, relative to that of a single elementary arithmetic operation, of evaluating $\varphi$ and generating a sample from $\mu$, respectively, and the rightmost $N$ and $N^K_\times$ terms account for the cost of computing the corresponding sample average once all evaluations of $\varphi$ are carried out. It follows that $\mu^N_\times(\varphi)$ is (asymptotically) at least as computationally efficient as $\mu^N(\varphi)$ if and only if the ratio of their respective costs is no smaller than one or, after some re-arranging,
\begin{equation}\label{eq:perform}\frac{\sigma^2(\varphi)}{\sigma^2_{\times}(\varphi)}\geq\frac{(\sigma^2_{\times}(\varphi)/{\sigma^2})^{K-1}C_r+1}{C_r+1},\end{equation}
where $C_r:=(C_\varphi+1)/C_X$ denotes the relative cost of evaluating the test function and drawing samples. Our first observation here is that, because $\sigma^2(\varphi)\geq \sigma^2_{\times}(\varphi)$ (Corollary~\ref{cor:varbound}), the above is always satisfied in the limit $C_r\to0$. This corresponds the case where the cost of acquiring the samples dwarfs the overhead of evaluating the sample averages (for instance, if the samples are obtained from long simulations or real-life experiments). If so, we do really want to make the most of the samples we have and product-form estimators help us to do so. Conversely, if samples are cheap to generate and the test function is expensive to evaluate (i.e. $C_r\to\infty$), then we are better off using the basic estimator.

To investigate the case where the costs of generating samples and evaluating the test function are comparable ($C_r\approx 1$), note that the variance approximation $\text{Var}(\mu^{N_\times}_\times(\varphi))\approx \sigma^2_\times(\varphi)/N_\times$ and, consequently, \eqref{eq:perform} are  valid only if $\sigma^2_{\times}(\varphi)>\sigma^2$. Otherwise, $N_\times=1$ and the product-form estimator simply equals $\varphi(X^1)$ with variance $\sigma^2(\varphi)$. In the high-dimensional (i.e.\ large $K$) case which is of particular interest, ~\eqref{eq:perform} then (approximately) reduces  to
\begin{equation}\label{eq:perform2}\frac{\sigma^2(\varphi)}{\sigma^2_{\times}(\varphi)}\geq\frac{1}{2}\left(\frac{\sigma^2_{\times}(\varphi)}{\sigma^2}\right)^{K-1}.\end{equation}
To gain insight into whether it is reasonable to expect the above to hold, we revisit Example~\ref{ex:K}.
\begin{example}\label{ex:perform}Setting once again our desired standard deviation to be proportional to the magnitude of the target average (i.e.\ $\sigma=\varepsilon\mmag{\mu(\varphi)}=\varepsilon|\rho(\psi)|^K$) and calling on~(\ref{eq:mfeu9anguiynguea},\ref{eq:mfeu9anguiynguea2}), we re-write~\eqref{eq:perform2} as
\begin{align*}
\frac{(1+CV^2)^K-1}{CV^2K}\geq \frac{(CV^2K\varepsilon^{-2})^{K-1}}{2}
 \Leftrightarrow  \frac{(1+CV^2)^K-1}{CV^{2K}}\geq \frac{\varepsilon^2}{2}\left(\frac{K}{\varepsilon^{2}}\right)^K.
\end{align*}
The expression shows that, in this full $\cal{O}(N^K)$ cost case, $\mu^N(\varphi)$ outperforms $\mu^N_\times(\varphi)$ in computational terms for large dimensions $K$ (and, even more so, for small relative tolerances $\varepsilon$). 
\end{example}

In summary, unless the cost of generating samples is significantly larger than that of evaluating $\varphi$, we expect the basic estimator to outperform the product-form one. Simply put, independent samples are more valuable for estimation than correlated permutations thereof. Hence, if independent samples are cheap to generate, then we are better off drawing further independent samples instead of permuting the ones we have. 

That is, unless one can find a way to evaluate the product-form estimator that does not require summing over all $N^K$ permutations. Indeed, the above analysis is out of place for Example~\ref{ex:perform} because, in this case, we can express the product-form estimator as the product 
\begin{equation}\label{eq:trick}\mu^N_\times(\varphi)=\prod_{k=1}^K\left(\frac{1}{N}\sum_{n=1}^N\psi(X_k^n)\right)=\prod_{k=1}^K\mu^N_k(\psi)\end{equation}
of the univariate sample averages $\mu^N_1(\psi),\dots,\mu^N_K(\psi)$ and evaluate each of these separately at a total $\cal{O}(KN)$ cost. 
Given that the number of samples required for $\mu^N_\times(\varphi)$ to yield a reasonable estimate scales linearly with dimension (Example~\ref{ex:K}), it follows that the cost incurred by computing such an estimate scales quadratically with dimension. In the case of $\mu^N(\varphi)$, the number of samples required, and hence the cost, scales exponentially with dimension; making the product-form estimator the clear choice for this simple case. This type of trick significantly expands the usefulness of product-form estimators, as we see in the following section.

\subsection{Efficient computation}\label{sec:lincost}
Recall our starting example from Section~\ref{sec:intro}. In that case, the product-form estimator trivially breaks down into the product of $K$ sample averages~\eqref{eq:intrvar} and, consequently, we can evaluate it in $\cal{O}(KN)$ operations. We can exploit this trick whenever the test function possesses product-like structure: if $\varphi$ is a  sum
\begin{equation}\label{eq:soptest}\varphi=\sum_{j=1}^J\varphi^j\enskip\text{of products}\enskip \varphi^j:=\prod_{k=1}^{K}\varphi_k^j\end{equation}
of univariate functions $(\varphi_k^j:S_k\to\r)_{j\in[J],k\in[K]}$, the product-form estimator decomposes into a sum of products (SOP) of univariate averages, 
$$\mu_\times^N(\varphi)=\sum_{j=1}^J\prod_{k=1}^K\mu_k^N(\varphi_k^j),$$
where
$$\mu_k^N(\varphi_k^j):=\frac{1}{N}\sum_{n=1}^N\varphi_k^j(X^n_k)\enskip\forall j\in[J],\enskip k\in [K],$$
and we are able to evaluate $\mu_\times^N(\varphi)$ in $\cal{O}(KN)$ operations. (Of course, `univariate' need not mean that the function is defined on $\r$ and we can be strategic in our choice of component spaces $S_1,\dots,S_K$; e.g.\ if $\varphi(x_1,x_2,x_3)=f(x_1,x_2)g(x_3)$ for some functions $f:\r^2\to\r$ and $g:\r\to\r$, we could pick $K:=2$, $S_1:=\r^2$, and $S_2:=\r$.)
In these cases,  the use of product-form estimators amounts to nothing more than a dimensionality-reduction technique: we exploit the independence of the target to express our $K$-dimensional integral  in terms of an SOP of one-dimensional integrals,
$$\mu(\varphi)=\sum_{j=1}^J\prod_{k=1}^K\mu_k(\varphi_k^j)=:\text{sop}(\{\mu_k(\varphi^j_k)\}_{j\in [J],k\in [ K]}),$$
estimate each of these separately,
$$\mu_k(\varphi_k^j)\approx\mu_k^N(\varphi_k^j)\quad\forall j\in [J],\enskip k\in [ K],$$
and replace the one-dimensional integrals in the SOP  with their estimates to obtain an estimate for the $K$-dimensional integral:
$$\mu(\varphi)\approx \text{sop}(\{\mu_k^N(\varphi^j_k)\}_{j\in [J],k\in [ K]})=\mu^N_\times(\varphi).$$
By so exploiting the structure in $\mu$ and $\varphi$, the product-form estimator achieves a lower variance than the standard estimator (Corollary~\ref{cor:varbound}). Moreover, evaluating each  univariate sample average $\mu_k^N(\varphi_k^j)$ requires only $\cal{O}(N)$ operations and, consequently the computational complexity of $\mu^N_\times(\varphi)$ is $\cal{O}(KN)$. 
The running time can be further reduced by calculating the univariate sample averages in parallel.

Similar considerations apply if the test function $\varphi$ is a  product of  low-dimensional functions (and sums thereof) instead of univariate ones, e.g.\ $\varphi(x)=\prod_{i=1}^I\varphi_i((x_k)_{k\in A_i})$ for a collection of factors $\varphi_1,\dots,\varphi_I$ with arguments indexed by subsets $A_1,\dots,A_I$ of $[K]$. As with the SOP case, one should aim to swap as many summation and product signs in
$$\mu_\times^N(\varphi)=\frac{1}{N^K}\sum_{n_1=1}^N\dots\sum_{n_K=1}^N\prod_{i=1}^I\varphi_i((X_{k}^{n_k})_{k\in A_i})$$
as the factors permit. Exactly how best to do this is obvious for simple situations such as that in Example~\ref{ex:hier} in Section~\ref{sec:is}. For more complicated ones, we advice using the `variable elimination' algorithm (e.g.\ \cite[Chapter~9]{Koller2009}) commonly employed for inference in discrete graphical models. The complexity of the resulting procedure essentially depends on the order in which one attempts the swapping (however, it is easy to find bounds thereon, for instance, it is bounded below by both the maximum cardinality of $A_1,\dots,A_I$ and half the length of the longest cycle in $\varphi$'s factor graph). While finding the ordering with lowest complexity for general partially-factorized $\varphi$ itself proves to be a problem whose worst-case complexity is exponential in $K$, good suboptimal orderings can often be found using cheap heuristics (cf.\ \cite[Section~9.4.3]{Koller2009}).

For general   $\varphi$ lacking any sort of product structure, we are sometimes able to extend the  linear-cost approach by approximating $\varphi$ with SOPs (e.g.\ using truncated Taylor expansions for analytic $\varphi$). The idea is that, if $\varphi\approx \varphi_{sop}$ for some SOP $\varphi_{sop}$, then 
$$\mu(\varphi)\approx\mu(\varphi_{sop}),\quad \text{Var}(\mu_\times^N(\varphi))\approx \text{Var}(\mu_\times^N(\varphi_{sop})),$$
and we can use $\mu^N_\times(\varphi_{sop})\approx \mu(\varphi_{sop})$ as a linear-cost estimator for $\mu(\varphi)$ without significantly affecting the variance reduction. This of course comes at the expense of introducing a bias in our estimates, albeit one that can often be made arbitrarily small by using more and more refined approximations (these biases may in principle be removed using multi-level randomization~\cite{Mcleish2010,Rhee2015}). The choice of approximation quality itself proves to be a balancing act as more refined approximations typically incur higher evaluation costs. If these costs are high enough, then any potential computational gains afforded by the reduction in variance are lost. In summary, this SOP approximation approach is most beneficial for test functions   (a) that are effectively approximated by SOP functions (so that the bias is low), (b) whose SOP approximations are relatively cheap to compute (so that the cost is low), and (c) that have a high-dimensional  product-form component to them (so that the variance reduction is large, cf.\ Section~\ref{sec:statseff}). In these cases, the gains in performance can be substantial as illustrated by the following toy example.
\begin{example}\label{ex:taylor}
Let $\mu_1, \dots, \mu_K$ be uniform distributions on the interval $[0, a]$ of length $a>1$ and consider the test function $\varphi(x):=e^{x_{1}\dots x_K}$. The integral can be expressed in terms of the generalized hypergeometric function $\pFq{p}{q}$,
\begin{align*}
\mu(\varphi) &= \sum_{j=0}^\infty\frac{\mu_1(x_1^j)\dots \mu_K(x_K^j)}{j!}= \sum_{j=0}^\infty\frac{1}{j!}\left[ \frac{a^j}{(j+1)}\right]^K=\pFq{K}{K}(1,\dots, 1; 2, \dots, 2; a^K),
\end{align*}
and grows super-exponentially with the dimension $K$ (see Fig.\ \ref{fig:taylor}a).  Because
\begin{align*}\varphi(x)=e^{x_{1}\dots x_K}&\approx\sum_{j=0}^J\frac{[x_{1}\dots x_K]^j}{j!}=1+\sum_{j=1}^J\frac{x_1^j\dots x_K^j}{j!}=:\varphi_J(x)\end{align*}
for large enough truncation cutoffs $J$, we have that
$$\mu^N_{\times}(\varphi)\approx\mu^N_{ \times}(\varphi_J)=1+\sum_{j=1}^{J}\frac{\mu_1^N(x_1^{j})\dots\mu_K^N(x_K^{j})}{j!}.$$
Using $\mu^N_{\times}(\varphi_J)$ instead of $\mu^N_{\times}(\varphi)$ as an estimator for $\mu(\varphi)$, we lower the computational cost from $\cal{O}(N^K)$ to $\cal{O}(K N)$. In exchange, we introduce a bias:
\begin{align*}\Ebb{\mu^N_{\times}(\varphi_J)}-\mu(\varphi)&=\mu(\varphi_J)-\mu(\varphi)=\mu(\varphi_J-\varphi)=\mu\left(\sum_{j=J+1}^\infty\frac{x_1^j\dots x_K^j}{j!}\right)\\
&=\sum_{j=J+1}^\infty\frac{\mu_1(x_1^j)\dots\mu_K(x_K^j)}{j!}=\sum_{j=J+1}^\infty\frac{1}{j!}\left[\frac{a^j}{j+1}\right]^K.
\end{align*}
As $\sum_{j=J+1}^\infty\frac{a^{jK}}{j!}=o(a^{JK}/J!)$, the bias decays super-exponentially with the cutoff $J$, at least for sufficiently large $J$. In practice, we found it to be significant for $J$s smaller than $0.8a^K$ and negligible for $J$s larger than $1.2a^K$ (Fig.\ \ref{fig:taylor}b). In particular, the cutoff $J$ necessary for $\mu^N_{\times}(\varphi_J)$ to yield estimates with small bias grows exponentially with the dimension $K$.

\begin{figure*}
    \begin{center}
    \includegraphics[width=1\textwidth]{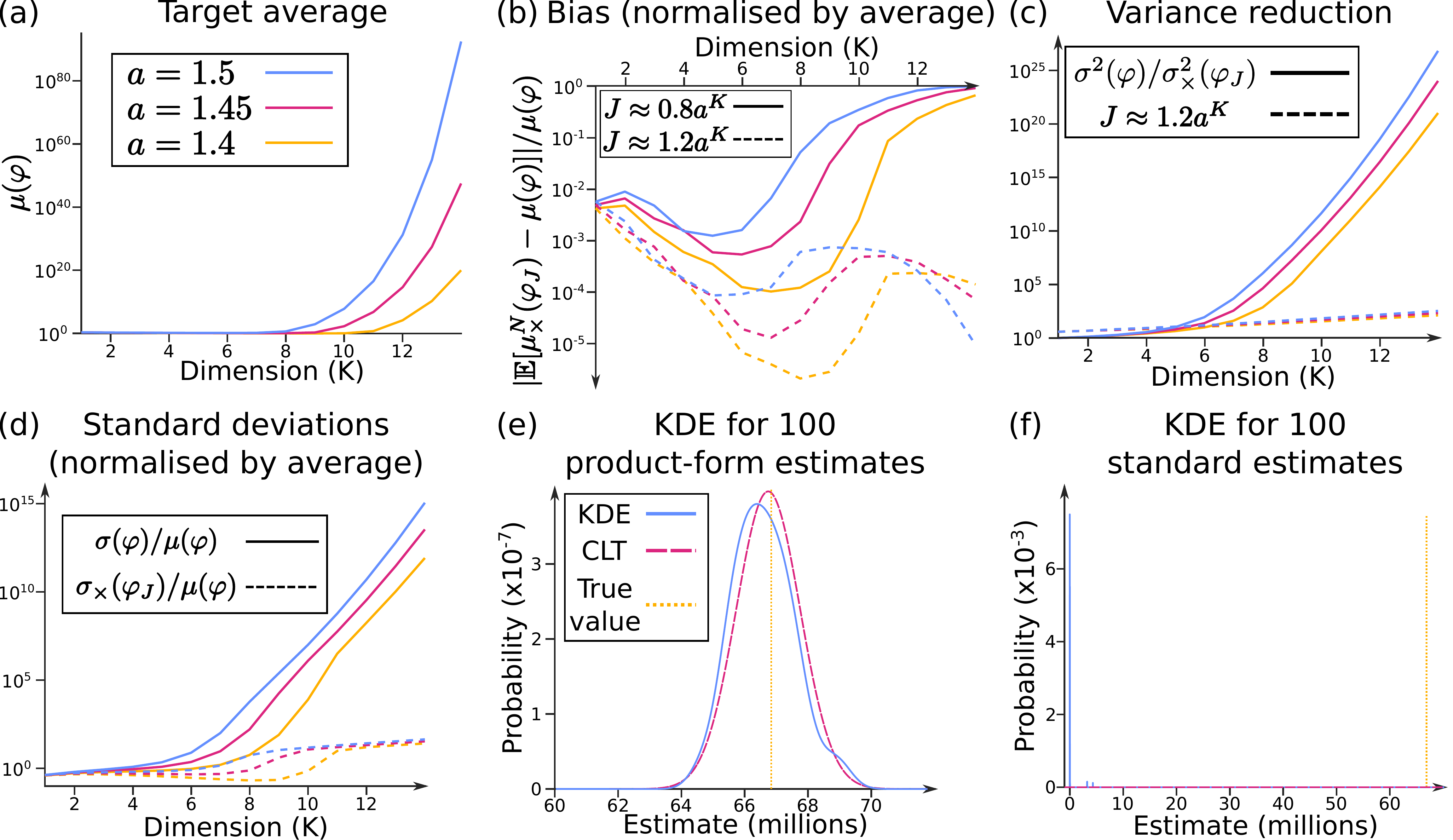}
    \vspace{-30pt}
    \end{center}

\caption{\textbf{(a--d)} Plots generated for three values of $a$:  $a=1.4$ (blue), $a=1.45$ (magenta), and $a=1.5$ (yellow). \textbf{(a)} Target average $\mu(\varphi)$ as a function of dimension $K$. \textbf{(b)} Bias of product-form estimator (normalized by target average) as a function of  $K$ with truncation cut-offs $J=\lceil 0.8a^K\rceil+2$ (solid) and $J=\lceil 1.2a^K\rceil+2$ (dashed). We added the $+2$ to avoid trivial cutoffs for low values of $a^K$.  \textbf{(c)} Ratio of asymptotic variances $\sigma^2(\varphi)/\sigma^2_\times(\varphi_J)$ (solid) with $J=\lceil 1.2a^K\rceil+2$ (dashed) as a function of $K$.  \textbf{(d)} Asymptotic standard deviation (normalized by target average) for conventional (solid) and biased product-form (dashed, with $J=\lceil 1.2a^K\rceil+2$) estimators as a function of $K$. \textbf{(e)} Kernel density estimator with plug-in bandwidth~\cite{Wand1994}  (blue) obtained with $a=1.5$, $K=10$, $J=70$, and $100$ repeats of  $\mu^N_\times(\varphi_J)$   each involving $N=10^6$ samples is a good match to the corresponding sampling distribution (magenta) predicted by the CLT in Theorem~\ref{thrm:classicalK}. Comparing with the target average (yellow), we find a mean absolute error across repeats of $6.73\times 10^5\approx 1\%\mu(\varphi)$. \textbf{(f)} As in (e) but for $\mu^N(\varphi)$. This time, the predicted sampling distribution is extremely wide (with a standard deviation of $6.7\times10^{14}$) and a poor match to the kernel density estimator (almost a Dirac delta close to zero). The mean absolute error is $6.67\times 10^7\approx \mu(\varphi)$. The estimator's failure stems from the extreme rarity of samples $X^n$ achieving very large values of $\varphi(X^n)$ (i.e.\ those with  components that are all close to $a$). Because the components  are independent, they are extremely unlikely to simultaneously be close to $a$ and the aforementioned samples are not observed for realistic ensemble sizes $N$. The product-form estimator avoids this issue by averaging over each component separately.}\label{fig:taylor}
\end{figure*}

However, similar manipulations to those above reveal that
\begin{align*}
   \sigma^2(\varphi) =& \pFq{K}{K}\left(1,\dots, 1; 2, \dots, 2; 2a^{K}\right)- \pFq{K}{K}(1,\dots, 1; 2, \dots, 2; a^K)^2\\
   \sigma_{\times}^2(\varphi_J) =& K\sum_{i=0}^J\sum_{j=0}^J \frac{1}{i!j!}\frac{ij}{i+j+1}\left( \frac{a^{i+j}}{(i+1)(j+1)}\right)^K
\end{align*}
and we find that the variance reduction achieved by $\mu^N_\times(\varphi_J)$ far outpaces the growth in  $K$ of the cutoff (and, thus, the computational cost of $\mu^N_\times(\varphi_J)$) necessary to achieve a small bias (Fig.\ \ref{fig:taylor}c).
Indeed, the asymptotic-standard-deviation-to-mean ratio, $\sigma(\varphi)/\mu(\varphi)$, rapidly diverges with $K$ in the case of the standard estimator (Fig.\ \ref{fig:taylor}d, solid). In that of the biased product-form estimator, the ratio, $\sigma_\times(\varphi_J)/\mu(\varphi)$, also diverges with $K$ but at a much slower rate (Fig.\ \ref{fig:taylor}d, dashed). For this reason,  the number of samples necessary for obtain a, say, $1\%$ accuracy estimate of $\mu(\varphi)$ using $\mu^N_\times(\varphi_J)$ remains manageable for a substantially larger range of $a$s and $K$s than in the case of $\mu^N(\varphi)$, even after factoring in the extra cost required to evaluate $\mu^N_\times(\varphi_J)$ for $J$'s large enough that the bias is insignificant. For instance, with an interval length of $1.5$, ten dimensions, a cutoff of seventy, one million samples, and less than one minute of computation time suffices for $\mu^N_\times(\varphi_J)$ to produce a $1\%$ accuracy estimate of $\mu(\varphi)\approx 6.68\times 10^7$ (Fig.\ \ref{fig:taylor}e). Using the same one million samples and the standard estimator, we obtain very poor estimates (Fig.\ \ref{fig:taylor}f). Indeed, its asymptotic variance  equals $4.45\times 10^{29}$ and, so, we would need approximately $ 10^{18}$ samples to obtain $1\%$ accuracy estimates using $\mu^N(\varphi)$, something far beyond current computational capabilities.
\end{example}
\section{Extensions to non-product-form targets}\label{sec:extend} 
While interesting product-form distributions can be found throughout the applied probability literature (ranging from the stationary distributions of Jackson queues~\cite{Jackson1957,Kelly1979} and complex-balanced stochastic reaction networks~\cite{Anderson2010,Cappelletti2016} to the mean-field approximations used in variational inference~\cite{Ranganath2014,Blei2017}), most target distributions encountered in practice are not product-form. In this section, we demonstrate how to combine product-form estimators with other Monte Carlo methodology and expand their utility beyond the product-form case. 

We consider three simple extensions: one to targets that are absolutely continuous with respect to fully-factorized distributions (Section~\ref{sec:is}), resulting in a product-form variant of importance sampling (e.g.\ \cite[Chapter~8]{chopin2020}); another to targets that are absolutely continuous with respect to partially-factorized distributions (Section~\ref{sec:ppf}), resulting in a product-form version of importance sampling squared~\cite{Tran2013}; and a final one to targets with intractable densities arising from latent variable models (Section~\ref{sec:pMH}), resulting in a product-form variant of pseudo-marginal MCMC~\cite{Schmon2021}. In all cases, we show theoretically that the product-form estimators achieve smaller variances than their standard counterparts. Then, we investigate their  performance numerically by applying them to a simple hierarchical model (Section~\ref{sec:numex}).

Lastly, we mention here that a further extension, this time to targets that are mixtures of product-form distributions, can be found in  Appendix~\ref{sec:mop}. Because many distributions may be approximated with these mixtures, this extension potentially opens the door to tackling still more complicated targets (at the expense of introducing some bias). 
\subsection{Importance sampling}\label{sec:is}
Suppose that we are given an unnormalized (but finite) unsigned target measure $\gamma$ that is absolutely continuous with respect to the product-form distribution $\mu$ in Section~\ref{sec:motivation}, and let $w:=d\gamma/d\mu$ be the corresponding Radon-Nikodym derivative.
Instead of the usual important sampling (IS) estimator~\cite[Chapter~8]{chopin2020}, $\gamma^N(\varphi):=\mu^N(w\varphi)$ with $\mu^N$ as in~\eqref{eq:upn},   for $\gamma(\varphi)$, we consider its product-form variant, $\gamma^N_\times(\varphi):=\mu_\times^N(w\varphi)$ with $\mu^N_\times$ as in~\eqref{eq:upnx}. 
The results of  Section~\ref{sec:motivation} immediately give us the following:
\begin{corollary}\label{cor:is}If $\varphi$ is $\gamma$-integrable, then $\gamma^N_\times(\varphi)$ is an unbiased estimator for $\gamma(\varphi)$. If, furthermore, $w\varphi$ belongs to $L^2_\mu$ , then $\gamma^N_\times(\varphi)$ is strongly consistent, asymptotically normal,  and its finite sample and asymptotic variances are bounded above by those of $\gamma^N(\varphi)$: 
\begin{align*}
\text{Var}(\gamma^N_{\times}(\varphi))=\text{Var}(\mu^{N}_\times(w\varphi))&\leq \text{Var}(\mu^{N}(w\varphi))=\text{Var}(\gamma^N(\varphi))\quad\forall N>0,\\
\sigma^2_{\gamma,\times}(\varphi)=\sigma^2_{\times}(w\varphi)&\leq\sigma^2(w\varphi)=\sigma^2_{\gamma}(\varphi),
\end{align*}
where $\text{Var}(\mu^{N}_\times(w\varphi))$ and $\sigma^2_{\times}(w\varphi)$ are as in Theorem~\ref{thrm:classicalK}.
\end{corollary}

\begin{proof}Replace $\varphi$ with $w\varphi$  in Theorem~\ref{thrm:classicalK} and Corollary~\ref{cor:varbound}.\end{proof}
Corollary~\ref{cor:is} tells us that $\gamma^N_\times(\varphi)$ is more statistically efficient than the conventional IS estimator $\gamma^N(\varphi)$ regardless of whether the target $\gamma$ is product-form or not. In a nutshell, $\mu_\times^N$  is a better approximation to the proposal $\mu$ than $\mu^N$   and, consequently, $\gamma^N_\times(dx)=w(x)\mu_\times^N(dx)$ is a better approximation to $\gamma(dx)=w(x)\mu(dx)$ than $\gamma^N(dx)=w(x)\mu^N(dx)$. Indeed, by constructing all $N^K$ permutations of the tuples $X^1,\dots,X^N$, we reveal other areas of similar $\mu$ probability and further explore the state space. This can be particularly useful when the proposal and target are mismatched as it can  amplify the number of tuples landing in the target's high probability regions (i.e.\ achieving high weights $w$) and, consequently, substantially improve the quality of the finite sample approximation (Figure~\ref{fig:ldlike2}).

\begin{figure*}[h!]
    \begin{center}
    \includegraphics[width=1\textwidth]{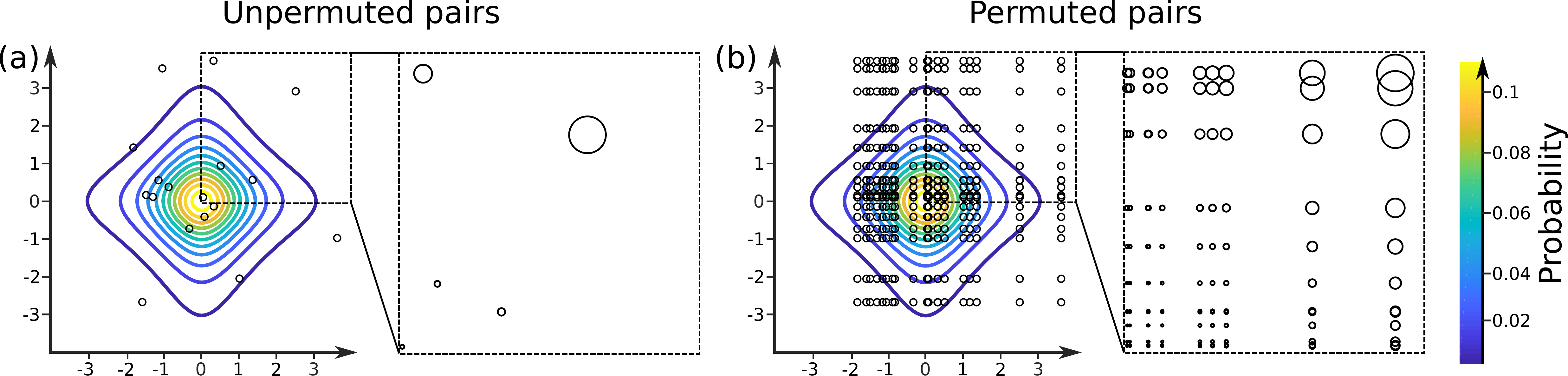}
    \vspace{-30pt}
    \end{center}
\caption{\textbf{Product-form approximations improve state space exploration.} The target, a uniform distribution on $[0,4]^2$ (dashed square), and the proposal, the product of two student-t distributions  with $1.5$ degrees of freedom (contours), are mismatched. Consequently, only $5$ of $20$ pairs independently drawn from the proposal land within the target's support  (\textbf{a}) and the corresponding weighted sample approximation (\textbf{a, inset,} dot diameter proportional to sample weight) is poor. By permuting these pairs, we improve the coverage of the state space (\textbf{b}), increase the number of pairs lying within the target's support, and obtain a much better weighted sample approximation (\textbf{b, inset,} dot diameter proportional to sample weight).}\label{fig:ldlike2}
\end{figure*}

\sloppy Similarly, the self-normalized version $\pi^N_\times(\varphi):=\gamma^N_\times(\varphi)/\gamma^N_\times(S)$
of the product-form IS estimator $\gamma^N_\times(\varphi)$ is a consistent and asymptotically normal estimator for averages $\pi(\varphi)$ with respect to the normalized target $\pi:=\gamma/\gamma(S)$. As in the case of the standard self-normalized importance sampling (SNIS) estimator $\pi^N(\varphi):=\gamma^N(\varphi)/\gamma^N(S)$, the ratio in $\pi^N_\times(\varphi)$'s definition introduces an $\cal{O}(N^{-1})$ bias 
and stops us from obtaining   analytical expression for the finite sample variance (that the bias is $\cal{O}(N^{-1})$ follows from an argument similar to that given for standard SNIS  in~\cite[p.35]{liu2001monte} and requires making assumptions on the higher moments of $\varphi(X^1)$). Otherwise, $\pi^N_\times(\varphi)$'s theoretical properties are analogous to those of the product-form estimator  $\mu^N_\times(\varphi)$ and its importance sampling extension $\gamma^N_\times(\varphi)$:

\begin{corollary}\label{cor:snis}
If $w\varphi$ belongs to $L^2_\mu$, then $\pi^N_\times(\varphi)$ is strongly consistent, asymptotically normal,  and its asymptotic variance is bounded above by that of $\pi^N(\varphi)$: 
\begin{align*}
\sigma^2_{\pi,\times}(\varphi)&=\sigma^2_{\times}(w^\pi[\varphi-\pi(\varphi)])\leq\sigma^2(w^\pi[\varphi-\pi(\varphi)])=\sigma^2_{\pi}(\varphi),
\end{align*}
where $w^\pi$ denotes the normalized weight function $w/\gamma(S)$ and $\sigma^2_{\times}(w^\pi[\varphi-\pi(\varphi)])$ is as in Theorem~\ref{thrm:classicalK}.
\end{corollary}

\begin{proof}\sloppy Given Theorem~\ref{thrm:classicalK} and Corollary~\ref{cor:varbound}, the arguments here follow closely those for standard SNIS. In particular, because  $\pi^N_\times(\varphi)=\gamma^N_\times(\varphi)/\gamma^N_\times(S)=\mu^N_\times(w\varphi)/\mu^N_\times(w)$ and  $\mu(w)=\gamma(S)$,
\begin{align*}\pi^N_\times(\varphi)-\pi(\varphi)&=\frac{\mu^N_\times(w\varphi)}{\mu^N_\times(w)}-\pi(\varphi)=\frac{\mu(w)}{\mu^N_\times(w)}\mu^N_\times\left(\frac{w[\varphi-\pi(\varphi)]}{\gamma(S)}\right)=\frac{\mu(w)}{\mu^N_\times(w)}\mu^N_\times(w^\pi[\varphi-\pi(\varphi)]).\end{align*}
Given that $\mu^N_\times(w)$ tends to $\mu(w)$ almost surely  (and, hence, in probability) as $N$ approaches infinity (Theorem~\ref{thrm:classicalK}), the strong consistency and asymptotic normality of $\pi^N_\times(\varphi)$ then follow from those of $\mu^N_\times(w^\pi[\varphi-\pi(\varphi)])$ (Theorem~\ref{thrm:classicalK})  and   Slutsky's theorem. The asymptotic variance bound follows from that in Corollary~\ref{cor:varbound}.
\end{proof}

This type of approach is best suited for targets $\pi$ possessing at least some product structure. The structure manifest itself in partially-factorized weight functions $w$ and substantially lowers the evaluation costs of $\gamma^N_\times(\varphi)$ and $\pi^N_\times(\varphi)$ for simple test functions $\varphi$, as the following example illustrates.
\begin{example}[A simple hierarchical model]\label{ex:hier}
Consider the following basic hierarchical model:
\begin{align}\label{eq:hierarchical}
Y_{k} \sim\cal{N}(X_k,1)\quad  X_{k} \sim\cal{N}(0,\theta),\quad \forall k\in[K].
\end{align}
It has a single unknown parameter, the variance $\theta$ of the latent variables $X_1,\dots,X_K$, which we infer using a Bayesian approach. That is, we choose a prior $p(d\theta)$ on  $\theta$ and draw inferences from the corresponding posterior,
\begin{equation}\label{eq:picid}\pi(d\theta,dx):=p(d\theta,dx|y)\propto p(d\theta)\prod_{k=1}^K\cal{N}(y_k;x_k,1)\cal{N}(dx_k;0,\theta)=:\gamma(d\theta,dx),\end{equation}
where $y=(y_1,\dots,y_K)$ denotes the vector of observations. For most priors, no analytic expressions for the normalizing constant can be found and we are forced to proceed  numerically. One option is to choose the proposal
\begin{equation}\label{eq:ISprop}\mu(d\theta,dx):= p(d\theta)\prod_{k=1}^K\cal{N}(dx_k;0,1),\end{equation}
in which case
$$w_{IS}(\theta,x):=\frac{d\gamma}{d\mu}(\theta,x)=\prod_{k=1}^K\frac{\cal{N}(y_k;x_k,1)\cal{N}(x_k;0,\theta)}{\cal{N}(x_k;0,1)}.$$
(Note that, were we to be using standard IS instead of product-form variant, the proposal 
\begin{equation}\label{eq:naturalprop}\mu(d\theta,dx):=p(d\theta)\prod_{k=1}^K\cal{N}(dx_k;0,\theta)\end{equation} 
would be the natural choice, a point we return to after the example.) Hence, to estimate the normalizing constant  or any integral w.r.t.\ to a univariate marginal of the posterior, we need to draw samples from $\mu$ and evaluate the product-form estimator $\mu^N_\times(\varphi)$ for a test function of the form $\varphi(\theta,x)=f(\theta)\prod_{k=1}^Kg_k(\theta,x_k)$, the cost of which totals $\cal{O}(KN^2)$ operations because
$$\mu^N_\times(\varphi)=\frac{1}{N}\sum_{m=1}^Nf(\theta^m)\prod_{k=1}^K\left(\frac{1}{N}\sum_{n_k=1}^Ng_k(\theta^m,x_k^{n_k})\right).$$

We return to this in Section~\ref{sec:numex}, where we will make use of the following expression for the (unnormalized) posteriors's $\theta$-marginal available due to the Gaussianity in~\eqref{eq:hierarchical}: 
\begin{equation}\label{eq:inverse-gamma} \gamma(d\theta)=p(d\theta)\prod_{k=1}^K\cal{N}(y_k;0,\theta+1).\end{equation}
Clearly, the above expression opens the door to simpler and more effective methods for  computing integrals with respect to this marginal than estimators targeting the full posterior.  
However, the estimators we discuss can be applied analogously to the many commonplace hierarchical models (e.g.\ see \cite{Gelman2006,Gelman2006a,Koller2009,hoffman2013stochastic,Blei2003} and the many references therein) for which such expressions are not available.
\end{example}

When applying IS, or extensions thereof like SMC, one should choose the proposal to be as close as possible to the target (e.g.\ see~\cite{Agapiou2017}). In this regard, the product-form IS approach is not entirely satisfactory for the above example: by definition, the  proposal must be fully factorized while the target, $\pi$ in~\eqref{eq:picid}, is only partially so (the latent variables are independent only when conditioned on the parameter variable). As we show in the next section, it is straightforward to adapt this product-form IS approach to match such partially-factorized targets.

\subsection{Partially-factorized targets and proposals}\label{sec:ppf}
Consider a target or proposal $\mu$ over a product space $(\Theta\times S,\cal{T}\times\cal{S})$ with the same partial product structure as the target in Example~\ref{ex:hier}:
\begin{equation}\label{eq:cidmu}\mu(d\theta,dx_1,\dots,dx_K)=\mu_0(d\theta)\prod_{k=1}^K\cal{M}_k(\theta,dx_k),\end{equation}
where, for each $k$ in $[K]$,  $\theta\mapsto \cal{M}_k(\theta,dx_k)$ denotes a Markov kernel mapping from $(\Theta,\cal{T})$ to $(S_k,\cal{S}_k)$. Suppose that we are given $M$ i.i.d.\ samples $\theta^1,\dots,\theta^M$ drawn from $\mu_0$ and, for each of these,  $N$ (conditionally) i.i.d.\ samples $X^{m,1},\dots,X^{m,N}$ drawn from the product kernel $\cal{M}(\theta,dx):=\prod_{k=1}^K\cal{M}_k(\theta,dx_k)$ evaluated at $\theta^m$.
Given a test function $\varphi$ on $\Theta\times S$, consider the following `partially product-form' estimator for $\mu(\varphi)$:
\begin{equation}\label{eq:ppf}\mu_{\times}^{M,N}(\varphi):=\frac{1}{M}\sum_{m=1}^M\left(\frac{1}{N^K}\sum_{\bm{n}\in[N]^K}\varphi(\theta^m,X^{m,\bm{n}})\right)=\frac{1}{MN^K}\sum_{m=1}^M\sum_{\bm{n}\in[N]^K}\varphi(\theta^m,X^{m,\bm{n}})\end{equation}
for all $M,N>0$. It is well-founded (for simplicity, we only consider the estimator's asymptotics as $M\to\infty$ with $N$ fixed, but other limits can be studied by combining the approaches in Appendix~\ref{app:classicalproof} with those in the proof of Theorem~\ref{thrm:classicalK2}):
\begin{theorem}
\label{thrm:classicalK2}
If $\varphi$ is $\mu$-integrable with $\mu$ as in~\eqref{eq:cidmu}, then $\mu^{M,N}_\times(\varphi)$ in~\eqref{eq:ppf} is unbiased and strongly consistent: for all $N>0$,
\begin{align*}\Ebb{\mu^{M,N}_\times(\varphi)}=\mu(\varphi)\quad\forall M>0,\quad 
\lim_{M\to\infty}\mu^{M,N}_\times(\varphi)= \mu(\varphi)\enskip\textrm{almost surely.}\end{align*}
If, furthermore, $\varphi$ belongs to $L^2_\mu$, then $\cal{M}_{[K]\backslash A}(\varphi)$ belongs to $L^2_{\mu_0\otimes\cal{M}_A}$ for all subsets $A$ of $[K]$, where $\cal{M}_A(\theta,dx_A):=\prod_{k\in A}\cal{M}_k(\theta,dx_k)$, and the estimator is asymptotically normal: 
\begin{equation}\label{eq:ppfclt}
M^{1/2}[\mu^{M,N}_\times(\varphi)-\mu(\varphi)]\Rightarrow \cal{N}(0,\sigma^2_{\times,N}(\varphi))\enskip\text{as $M\to\infty$,}\quad\forall N>0,\end{equation}
where  $\Rightarrow$ denotes convergence in distribution and
$$\sigma_{\times,N}^2(\varphi):=\mu_0([\cal{M}\varphi-\mu(\varphi)]^2)+\sum_{\emptyset\neq A\subseteq [K]}\frac{1}{N^{\mmag{A}}}\sum_{B\subseteq A}(-1)^{\mmag{A}-\mmag{B}}\mu_0(\cal{M}_B[\cal{M}_{[K]\backslash B}\varphi-\cal{M}\varphi]^2).$$
For any $N,M>0$, the estimator's variance is given by $\text{Var}(\mu^{M,N}_\times(\varphi))=\sigma^2_{\times,N}(\varphi)/M$.
\end{theorem}
\begin{proof}See Appendix~\ref{app:classicalK2}.
\end{proof}
The partially product-form estimator $\mu^{M,N}_\times(\varphi)$ is more statistically efficient than its standard counterpart:\
\begin{corollary}\label{cor:ppf}For any  $\varphi$ belonging to $L^2_\mu$ and $N>0$,
$$\text{Var}(\mu_\times^{M,N}(\varphi))\leq \text{Var}(\mu^{M,N}(\varphi))\enskip \forall M>0,\quad \sigma_{\times,N}^2(\varphi)\leq \sigma_{N}^2(\varphi),$$
where $\mu^{M,N}(\varphi):=\frac{1}{MN}\sum_{m=1}^M\sum_{n=1}^N\varphi(\theta^n,X^{m,n})$  and $\sigma_{N}^2(\varphi)$ denotes its asymptotic (in $M$) variance.
\end{corollary}
\begin{proof}See Appendix~\ref{app:classicalK2}.
\end{proof}
In fact, modulo a small caveat (cf.\ Remark~\ref{rem:ppfcaveat} below), $\mu^{M,N}_\times(\varphi)$ yields  the best  unbiased estimates of $\mu(\varphi)$ achievable using only the knowledge that  $\mu$ is partially-factorized and $M$ i.i.d.\ samples drawn from $\mu_0\otimes\cal{M}^N$: a perhaps unsurprising fact given that it is the composition of two minimum variance unbiased estimators (Theorem~\ref{thrm:mvu}).
\begin{theorem}
\label{thrm:mvuppf}Suppose that $\cal{T}$ contains all singleton sets (i.e.\ $\{\theta\}$ for all $\theta$ in $\Theta$). For any given measurable real-valued function $\varphi$ on $\Theta\times S$, $\mu_\times^{M,N}(\varphi)$ is a minimum variance unbiased estimator for $\mu(\varphi)$: if $f$ is a measurable real-valued function on $(\Theta \times S^N)^M$ such that
$$\Ebb{f((\theta^{m},X^{m,1},\dots,X^{m,N})_{m=1}^M)}=\mu(\varphi)$$
whenever $(\theta^{m},X^{m,1},\dots,X^{m,N})_{m=1}^M$ is an i.i.d.\ sequence drawn from $\mu_0\otimes \cal{M}^N$, 
for all partially-factorized $\mu=\mu_0\otimes \cal{M}$ on $\Theta\times S$ satisfying $\mu(\mmag{\varphi})<\infty$ and  
\begin{equation}\label{eq:continuous}
\mu_0(\{\theta\})=0\quad\forall \theta\in\Theta,
\end{equation} 
then 
$$\text{Var}(f((\theta^{m},X^{m,1},\dots,X^{m,N})_{m=1}^M))\geq\text{Var}(\mu^{M,N}_\times(\varphi)).$$
\end{theorem}
\begin{proof}See Appendix~\ref{app:ppfmvu}.\end{proof}
\begin{remark}[The importance of~\eqref{eq:continuous}]\label{rem:ppfcaveat} Consider the extreme scenario that $\mu_0$ is a Dirac delta at some $\theta^*$, so that $\theta^1=\dots=\theta^M=\theta^*$ with probability one and
$$\mu_{\times}^{M,N}(\varphi)=\frac{1}{M}\sum_{m=1}^M\frac{1}{N^K}\sum_{\bm{n}\in[N]^K}\varphi(\theta^*,X^{m,\bm{n}})\quad\text{almost surely}.$$
In this case, we are  clearly better off (at least in terms estimator variance) stacking all of our $X$ samples into one big ensemble and replacing the partially product-form estimator with the  (fully) product-form estimator,
$$\mu_{\times}^{MN}(\varphi)=\frac{1}{(MN)^K}\sum_{\bm{l}\in[MN]^K}\varphi(\theta^*,\tilde{X}^{\bm{l}}),$$
where $(\tilde{X}^{l})_{l\in[MN]}$ denotes $(X^{m,n})_{m\in[M],n\in[N]}$ in vectorized form (indeed Theorem~\ref{thrm:mvu} implies that $\mu_{\times}^{MN}(\varphi)$ is a minimum variance unbiased estimator in this situation). More generally, note that, because
\begin{align*}\mu^2_0(\{\theta^1=\theta^2\})=\int 1_{\{\theta^1=\theta^2\}}\mu^2_0(d\theta^1,d\theta^2)=\int\left(\int 1_{\{\theta^1=\theta^2\}}\mu_0(d\theta^1)\right)\mu_0(d\theta^2)=\int\mu_0(\{\theta\})\mu_0(d\theta),\end{align*}
$\mu_0$ not possessing atoms, i.e.\ \eqref{eq:continuous}, is equivalent to $\mu^2_0(\{\theta^1=\theta^2\})=0$. It is then straightforward to argue that~\eqref{eq:continuous} is equivalent to the impossibility of several $\theta^m$ coinciding or, in other words, to
\begin{equation}\label{eq:continuous1}
\mu_0^M(\{\theta^i\neq\theta^j\enskip\forall i\neq j\})=1.
\end{equation} 
Were this not to be the case, the estimator in~\eqref{eq:ppf} would not posses the MVU property. To recover it, we would need to amend the estimator as follows: `if several $\theta^m$s take the same value, first stack their corresponding $X^{m,1},\dots,X^{m,N}$ samples, and then apply a product-form estimator to the stacked samples'. However, to not overly complicate this section's exposition and Theorem~\ref{thrm:mvuppf}'s proof, we restricted ourselves to distributions satisfying~\eqref{eq:continuous}. 
\end{remark}
We are now in a position to revisit Example~\ref{ex:hier} and better adapt the proposal to the target. This leads to a special case of an algorithm
  known as `importance sampling squared' or `IS$^2$'~\cite{Tran2013}:
\begin{example}[A simple hierarchical model, revisited]\label{ex:hier1} Consider again the model in Example~\ref{ex:hier}. Recall that our previous choice of proposal did not quite capture the conditional independence structure in the target $\pi$: the former was fully factorized while the latter is only partially so. It seems more natural to instead use the proposal in~\eqref{eq:naturalprop}, which is also easy to sample from but both mirrors $\pi$'s independence structure  and leads to further cancellations in the weight function  (in particular, it no longer depends on $\theta$):
$$w_{IS^2}(x):=\prod_{k=1}^K\cal{N}(y_k;x_k,1)=\frac{d\gamma}{d\mu}(\theta,x).$$
It follows that, to estimate the normalizing constant  or any integral w.r.t.\ to a univariate marginal of the posterior, we need to draw samples from $\mu$ and evaluate the partially product-form estimator $\mu^{M,N}_\times(\varphi)$ for a test function of the form $\varphi(\theta,x)=f(\theta)\prod_{k=1}^Kg_k(x_k)$. The total cost then reduces to  $\cal{O}(KMN)$, because
$$\mu^{M,N}_\times(\varphi)=\frac{1}{M}\sum_{m=1}^Mf(\theta^m)\prod_{k=1}^K\left(\frac{1}{N}\sum_{n_k=1}^Ng_k(X_k^{m,n_k})\right).$$
We also return to this is Section~\ref{sec:numex}.
\end{example}
\subsection{Grouped independence Metropolis-Hastings}\label{sec:pMH}
As a further example of how one may embed product-form estimators within more sophisticated Monte Carlo methodology and exploit the independence structure present in the problem, we revisit Beaumont's Grouped Independence Metropolis-Hastings (GIMH,~\cite{Beaumont2003}),  a simple and well-known pseudo-marginal MCMC sampler~\cite{Andrieu2009}. Like many of these samplers, it is intended to tackle targets whose densities cannot be evaluated pointwise but are marginals of higher-dimensional distributions   whose densities can be  evaluated pointwise.  Our inability to evaluate the target's density precludes us from directly applying the Metropolis-Hastings algorithm (MH, e.g.\ see~\cite[Chap.\ XIII]{Asmussen2007}) as we cannot compute the necessary  acceptance probabilities. For instance, in the case of a target $\pi(d\theta)$ on a space $(\Theta,\cal{T})$ and an MH proposal $Q(\theta,d\tilde{\theta})$ with respective densities $\pi(\theta)$ and  $Q(\theta,\tilde{\theta})$, we would need to evaluate 
$$1\wedge \frac{\pi(\tilde{\theta})Q(\theta,\tilde{\theta})}{\pi(\theta)Q(\tilde{\theta},\theta)}$$
where $\theta$ denotes the chain's current state and $\tilde{\theta}\sim Q(\theta,\cdot)$ the proposed move. GIMH instead replaces the intractable $\pi(\theta)$ and $\pi(\tilde{\theta})$ in the above with importance sampling estimates thereof: if $\pi(\theta,x)$  denotes the density of the higher-dimensional distribution $\pi(d\theta,dx)$ whose marginal is $\pi(d\theta)$, and $w(\theta,x):=\pi(\theta,x)/\cal{M}(\theta,x)$ for a given Markov kernel $\cal{M}(\theta,dx)$ with density $\cal{M}(\theta,x)$,
\begin{equation}\label{eq:acceptest}
\pi^N(\theta)=\frac{1}{N}\sum_{n=1}^Nw(\theta,X^n),\quad \pi^N(\tilde{\theta})=\frac{1}{N}\sum_{n=1}^Nw(\tilde{\theta},\tilde{X}^n),
\end{equation}
where  $X^{1},\dots,X^N$ and $\tilde{X}^1,\dots,\tilde{X}^N$ are i.i.d.\ samples respectively drawn from $\cal{M}(\theta,\cdot)$ and $\cal{M}(\tilde{\theta},\cdot)$. Key in Beaumont's approach is that the samples are recycled from one iteration to another: if $Z^1,\dots,Z^N$ and $\tilde{Z}^1,\dots\tilde{Z}^N$ denote the i.i.d.\ samples used in the previous iteration, then $(X^1,\dots,X^N):=(Z^1,\dots,Z^N)$ if the previous move was rejected and $(X^1,\dots,X^M):=(\tilde{Z}^1,\dots,\tilde{Z}^N)$ if it was accepted.

As explained in~\cite{Andrieu2009,Andrieu2015}, the algorithm's correctness does not require the density estimates to be generated by~\eqref{eq:acceptest}, only for them to be unbiased. In particular, if the estimates are unbiased, GIMH may be interpreted as an MH algorithm on an expanded state space with an extension of $\pi(d\theta)$ as its invariant distribution. Consequently, provided that the density estimator is suitably well-behaved, GIMH returns consistent and asymptotically normal estimates of the target under conditions comparable to those for standard MH algorithms (e.g.\ the GIMH chain is uniformly ergodic whenever the associated `marginal' chain is and the estimator is uniformly bounded~\cite{Andrieu2009}; see \cite{Andrieu2015} for further refinements). Consequently, if the kernel is product-form (i.e.\ $\cal{M}(\theta,dx)$ is product-form for each $\theta$), we may replace~\eqref{eq:acceptest} with their product-form counterparts:
\begin{equation}\label{eq:acceptestpf}
\pi^N_\times(\theta)=\frac{1}{N^K}\sum_{\bm{n}\in[N]^K}w(\theta,X^{\bm{n}}),\quad \pi^N_\times(\tilde{\theta})=\frac{1}{N^K}\sum_{\bm{n}\in[N]^K}w(\tilde{\theta},\tilde{X}^{\bm{n}}),
\end{equation}
where $K$ denotes the dimensionality of the $x$-variables (the unbiasedness follows from $X^{\bm{n}}$ and $\tilde{X}^{\bm{n}}$ having respective laws $\cal{M}(\theta,dx)$ and $\cal{M}(\tilde{\theta},d\tilde{x})$ for any $\bm{n}$ in  $[N]^K$). Thanks to the results in \cite{Andrieu2016}, it is straightforward to show that this choice leads to lower estimator variances, at least asymptotically:
\begin{corollary}\label{cor:pmh}Let $(\theta^{m,N})_{m=1}^\infty$ and $(\theta_\times^{m,N})_{m=1}^\infty$ be the GIMH chains generated using~\eqref{eq:acceptest} and \eqref{eq:acceptestpf}, respectively, and the same proposal $Q(\theta,d\theta)$. If $\varphi$ belongs to $L^2_\pi$, then 
$$\lim_{M\to\infty}\textrm{Var}\left(\frac{1}{\sqrt{M}}\sum_{m=1}^M\varphi(\theta^{m,N}_\times)\right)\leq \lim_{M\to\infty}\textrm{Var}\left(\frac{1}{\sqrt{M}}\sum_{m=1}^M\varphi(\theta^{m,N})\right)\quad\forall N>0.$$
\end{corollary}
\begin{proof}See Appendix~\ref{app:pmh}.
\end{proof}
Given the argument used in the proof, the results of \cite{Andrieu2016} (Theorem~10 in particular) imply much more than the variance bound in the corollary's statement. For instance, if the target is not concentrated on points, then the spectral gap of $(\theta^{m,N}_\times)_{m=1}^\infty$ is bounded below by that of $(\theta^{m,N})_{m=1}^\infty$. 
We finish the section by returning to our running example:
\begin{example}[A simple hierarchical model, re-revisited]\label{ex:hier2} Here, we follow~\cite[Section~5.1]{Schmon2021}. Consider once again the model in Example~\ref{ex:hier} and suppose we are interested only in the posterior's $\theta$-marginal $\pi(d\theta)$. Choosing 
$$\cal{M}(\theta,dx):=\prod_{k=1}^K\cal{N}(dx_k;0,\theta),$$
$w$ factorizes,
$$w_{GIMH}(\theta,x)=\frac{\pi(\theta,x)}{\cal{M}(\theta,x)}=p(\theta)\prod_{k=1}^K\cal{N}(y_k;x_k,1);$$
resulting in a evaluation cost of $\cal{O}(KN)$ for~(\ref{eq:acceptest},\ref{eq:acceptestpf}) and, regardless of which density estimates we use, a total cost of $\cal{O}(KMN)$ where $M$ denotes the number of steps we run the chain for. We return to this in the following section.
\end{example}
\subsection{Numerical comparison}\label{sec:numex}
Here, we apply the estimators discussed throughout Sections~\ref{sec:is}--\ref{sec:pMH} to the simple hierarchical model  introduced in Example~\ref{ex:hier} and we examine their performance. To benchmark the latter, we choose the prior to be conditionally conjugate to the model's likelihood: $p(d\theta)$ is the 
Inv-Gamma$(\alpha/2,\alpha\beta/2)$ distribution, in which case
\begin{align*}
X_k|y_k,\theta\sim \cal{N}\left(\frac{y_k}{\theta^{-1}+1},\frac{1}{\theta^{-1}+1}\right)\enskip\forall{k\in[K]},\quad
 \theta| y,X\sim \textrm{Inv-Gamma}\left(\frac{\alpha+K}{2},\frac{\alpha\beta+\sum_{k=1}^KX_k^2}{2}\right);
\end{align*}
and we can alternatively approximate the posterior, $\pi(d\theta,dx)$ in~\eqref{eq:picid}, using a Gibbs' sampler. Note that the above expressions are unnecessary for the evaluation of the estimators in Sections~\ref{sec:is}--\ref{sec:pMH}. To compare with standard methodology that also does not requires such expressions, we also approximate the posterior using Random Walk Metropolis (RWM) with the proposal variance tuned so that the mean acceptance probability (approximately) equals $25\%$. To keep the comparison honest, we run these two chains for $N^2$ steps and set $M=N$ for the estimators in Sections~\ref{sec:ppf}~and~\ref{sec:pMH}; in which case all estimators incur a similar $\cal{O}(KN^2)$ cost. We further fix $K:=100$, $\alpha:=1$, $\beta:=1$, and $N:=100$ and generate artificial observations $y_1,\dots,y_{100}$  by running~\eqref{eq:hierarchical} with $\theta:=1$.

Figure~\ref{fig:KDEs}   shows approximations to the posteriors's $\theta$-marginal $\pi(d\theta)$ obtained using a Gibbs sampler, RWM, IS (Section~\ref{sec:is}), IS$^2$ (Section~\ref{sec:ppf}), GIMH (Section~\ref{sec:pMH}), and the last three's product-form variants (PFIS, PFIS$^2$, and PFGIMH, respectively). In the cases of Gibbs, RWM, GIMH, and PFGIMH, we used a $20\%$ burn-in period and approximated the marginal with the empirical distribution of the $\theta$-components of the states visited by the chain. For GIMH and PFGIMH we also used a random walk proposal with its variance tuned so that the mean acceptance probability hovered around $25\%$. For IS, PFIS, IS$^2$, and PFIS$^2$, we used the proposals specified in Examples~\ref{ex:hier}~and~\ref{ex:hier1} and computed the approximations using 
\begin{align*}&\pi^{N^2}_{IS}(d\theta):=\frac{\sum_{n=1}^{N^2}w_{IS}(\theta^n,X^n)\delta_{\theta^n}}{\sum_{n=1}^{N^2}w_{IS}(\theta^n,X^n)},\quad 
&\pi^{N}_{PFIS}(d\theta):=\frac{\sum_{n=1}^N\left(\sum_{\bm{n}\in[N]^K}w_{IS}(\theta^n,X^{\bm{n}})\right)\delta_{\theta^n}}{\sum_{n=1}^N\sum_{\bm{n}\in[N]^K}w_{IS}(\theta^n,X^{\bm{n}})},\\
&\pi^{N,N}_{IS^2}(d\theta):=\frac{\sum_{m=1}^N\left(\sum_{n=1}^Nw_{IS^2}(X^{m,n})\right)\delta_{\theta^m}}{\sum_{m=1}^N\sum_{n=1}^Nw_{IS^2}(X^{m,n})},\quad
 &\pi^{N,N}_{PFIS^2}(d\theta):=\frac{\sum_{m=1}^{N}\left(\sum_{\bm{n}\in[N]^K}w_{IS^2}(X^{m,\bm{n}})\right)\delta_{\theta^m}}{\sum_{m=1}^{N}\sum_{\bm{n}\in[N]^K}w_{IS^2}(X^{m,\bm{n}})}.
\end{align*}
(Note that for IS, we are using $N^2$ samples instead of $N$ so that its cost is also $\cal{O}(KN^2)$.)

Our first observation is that the approximations produced by IS, IS$^2$, and GIMH are very poor. The first two exhibit severe weight degeneracy (in either case, a single particle had over $50\%$ of the probability mass and three had over $90\%$), something unsurprising given the target's moderately high dimension of $101$\footnote{One may wonder whether in the case of IS, the degeneracy could instead be due to our use of the proposal~\eqref{eq:ISprop} rather than the more natural choice~\eqref{eq:naturalprop}. It is not: the average W$_1$ distance and KS statistic (see Figure~\ref{fig:KDEs}'s caption for definitions) across $100$ replicates of the $\pi(d\theta)$'s approximation obtained using~\eqref{eq:naturalprop} and IS (with $N^2$ samples) were $32.7\%$ and $82.3\%$, respectively. In other words, a modest improvement over IS with proposal~\eqref{eq:ISprop} (compare with Table~\ref{table2}), but not  one sufficient to break the degeneracy: $83$  approximations (out of $100$) had at least $50\%$ of their mass concentrated in $2$ particles (out of $10,000$) and all but $7$ had over $80\%$ of their mass concentrated in $10$ particles.}. The third possesses a large spurious peak close to zero (with over $70\%$ of the mass) caused by large numbers of rejections in that vicinity. Replacing the i.i.d.\ estimators embedded within these algorithms with their product-form counterparts removes both the weight degeneracy and the spurious peak; and PFIS, PFIS$^2$, and PFGIMH return much improved approximations.   The best approximation is the one returned by the Gibbs sampler: an expected outcome given that the sampler's use of the conditional distributions makes it the estimator most `tailored' or `well-adapted' to the target. However, these distributions are not available for most models (precluding application of these samplers to such models) and   even just taking the, usually obvious, independence structure into account can make a substantial difference: the quality of the approximations returned by PFIS and PFIS$^2$ exceeds the quality of that returned by the common, or even default, choice of RWM. Note that this is the case even though the proposal variance in RWM was tuned, while that in the other two was simply set to $1$ (a reasonable choice given that $\theta=1$ was used to generate the data, but likely not the optimal one). In fact, for this simple model, it is easy to sensibly incorporate observations into the PFIS and PFIS$^2$ proposals (e.g.\ use $p(d\theta)\prod_{k=1}^K\cal{N}(dx_k;y_k,1)$ for PFIS and $p(d\theta)\prod_{k=1}^K\cal{N}(dx_k; y_k\theta[1+\theta]^{-1},\theta[1+\theta]^{-1})$ for PFIS$^2$) and potentially improve their performance.

\begin{figure*}[h!]
    \begin{center}
    \includegraphics[width=1\textwidth]{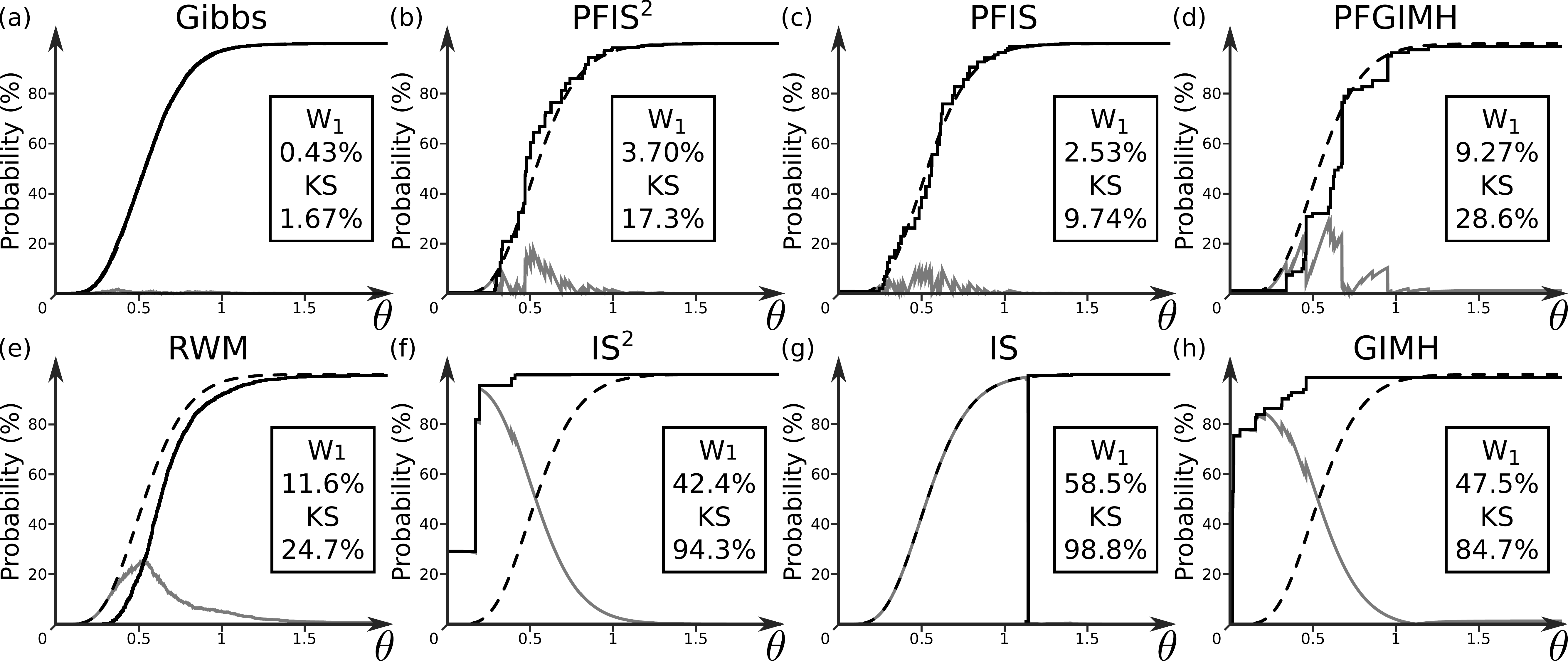}
    \vspace{-40pt}
    \end{center}
\caption{Empirical cumulative density functions for $\pi(d\theta)$ obtained using the eight approximations discussed in the text (black solid lines). As guides, we also plot a high-quality approximation $\pi_{REF}(d\theta)$ using~\eqref{eq:inverse-gamma} and quadrature, and the pointwise absolute difference between $\pi_{REF}$ and the eight approximations (grey lines). \textbf{Insets.} Wasserstein-1 distance (W$_1$) between $\pi_{REF}$ and the panel's approximation (area under the grey line, e.g.\ see \cite[p.64]{Shorack2009}) and corresponding Kolmogorov--Smirnov statistic (KS, maximum of the grey line).}\label{fig:KDEs}\vspace{-10pt}
\end{figure*}

To benchmark the approaches more thoroughly, we generated $R:=100$ replicates of the eight full posterior approximations and computed various error metrics (Tables~\ref{table2}~and~\ref{table22}). For the $\theta$-component, we used the high-quality reference approximation $\pi_{REF}$ described in Figure~\ref{fig:KDEs}'s caption to obtain the average (across repeats) W$_1$ distance and KS statistic (as described in the caption), and the average absolute error of the posterior mean and standard deviation estimates normalized by the true mean or standard deviation (i.e.\ $M_\theta^{-1}R^{-1}\sum_{r=1}^{R}\mmag{M^r_\theta-M_\theta}$ for the posterior mean estimates, where $M_\theta$ denotes the true mean  and $M^r_\theta$ the $r^{th}$ estimate thereof, and similarly for the standard deviation estimates). For the $x$-components, we instead used high-accuracy estimates for the component-wise means and standard deviations (obtained by running a Gibbs sampler  for $N^4=10^8$ steps) to compute the corresponding total absolute errors across replicates and components ($\sum_{k=1}^K\sum_{r=1}^{R}\mmag{M^{r}_{k}-M_{k}}$, where $M_k$ denotes the true mean  for the $k^{th}$ $x$-component and $M^r_k$ the $r^{th}$ estimate thereof, and similarly for the standard deviation estimates).

\begin{table}[!ht]
\begin{center}
\begin{tabular}{c|cccccccc}
&Gibbs&PFIS$^2$&{PFIS}&PFGIMH&{RWM}&IS$^2$&{IS}&GIMH\\\hline
W$_1$&$0.73\%$&$4.41\%$&$5.33\%$&$7.95\%$&$9.99\%$&$32.7\%$&$36.0\%$&$39.9\%$\\
KS&$1.93\%$&$17.9\%$&$19.2\%$&$24.9\%$&$23.3\%$&$82.3\%$&$79.9\%$&$71.8\%$\\
Mean error&$1.08\%$&$4.51\%$&$7.84\%$&$8.00\%$&$16.2\%$&$53.6\%$&$60.4\%$&$70.6\%$\\
Standard deviation error&$1.51\%$&$8.07\%$&$9.01\%$&$17.1\%$&$21.0\%$&$64.4\%$&$64.0\%$&$26.4\%$\\
\end{tabular}
\caption{Average-across repeats W$_1$ error and KS statistic for the approximations of $\pi(d\theta)$, and  average absolute errors for the corresponding mean and standard deviation estimates, obtained using each of the eight methods.}\vspace{-20pt}
\label{table2}
\end{center}
\end{table}

\begin{table}[!ht]
\begin{center}
\begin{tabular}{c|cccccc}
&Gibbs&PFIS$^2$&{PFIS}&{RWM}&IS$^2$&{IS}\\\hline
Mean&$56$&$169$&$511$&$1209$&$3295$&$5727$\\
Standard deviation&$40$&$124$&$332$&$663$&$2874$&$3340$\\
\end{tabular}
\caption{Total absolute error for the mean and standard deviation estimates of $\pi(dx)$'s univariate marginals. Note that no results are given for GIMH and PFGIMH since these algorithms directly target the $\theta$-marginal $\pi(d\theta)$.}\vspace{-20pt}
\label{table22}
\end{center}
\end{table}
Once again, the product-form estimators far outperformed their i.i.d.\ counterparts. Moreover, they perform just as well or better than RWM. PFIS$^2$'s estimates are  particularly accurate: a fact that does not surprise us given that its proposal has the same partially-factorized structure as the target, in this sense making it the best adjusted estimator to the problem. That is, best except for the Gibbs sampler that exploits the conditional distributions (encoding more information than this structure).  
We conclude with an interesting detail: PFIS$^2$ and PFIS perform similarly when approximating  the $\theta$-marginal (cf.\ Table~\ref{table2}), but PFIS$^2$ outperforms PFIS when approximating the latent variable marginals (cf.\ Table~\ref{table22}). This is perhaps not too surprising because, in the case of the $\theta$-marginal approximation, both PFIS$^2$ and PFIS  employ the same number $N$ of $\theta$-samples, while, in that of $k^{th}$ latent variable,  PFIS$^2$ uses $N^2$ $x_k$-samples but PFIS uses only $N$ such samples.
\section{Discussion}\label{sec:discussion}
The main message of this paper is that, when using Monte Carlo estimators to tackle problems possessing some sort of product structure, one should endeavour to exploit this structure and improve the estimators' performance. The resulting product-form estimators are not a panacea for the curse of dimensionality in Monte Carlo, but they are a useful  and sometimes overlooked tool in the practitioner's arsenal and make certain problems solvable when they otherwise would not be. More specifically,  whenever the target, or proposal, we are drawing samples from is product-form, these estimators achieve a smaller variance than their conventional counterparts. In our experience (e.g.\ Examples~\ref{ex:K}~and~\ref{ex:taylor}), the gap in variance grows at least exponentially with dimension whenever the integrand does not decompose into a sum of low-dimensional functions like in the trivial case~\eqref{eq:lintest}. For the reasons given  in Section~\ref{sec:statseff}, we expect the variance reduction to be further accentuated by targets that are `spread out' rather than highly peaked.

The gains in statistical efficiency come at a computational price: in the absence of exploitable structure in the test function, product-form estimators incur an $\cal{O}(N^K)$ cost limiting their applicability to $K<10$, while conventional estimators only carry an $\cal{O}(N)$ cost (although in practice the cost of obtaining reasonable estimates using the latter often scales poorly with $K$, with the effect hidden in the proportionality constant, e.g.\ Examples~\ref{ex:K}~and~\ref{ex:taylor}). 
Hence, for general test functions, product-form estimators are of most use when the variance reduction is particular pronounced or when samples are expensive to acquire (both estimators require drawing the same number $N$ of samples) or store (as, for example, when one employs physical random numbers and requires reproducibility \cite{owen2009}). In the latter case, product-form estimators enable us to extract the most possible from the samples we have gathered so far: by permuting the samples' components, the estimators artificially generate further samples. Of course, the more permutations we make, the more correlated our sample ensemble becomes and we get a diminishing returns effect that results in an $\cal{O}(N^{-1/2})$ rate of convergence instead of the $\cal{O}(N^{-K/2})$ rate we would achieve using $N^K$ independent samples. There is a middle ground here that remains unexplored: using $N<M<N^K$ permutations instead of all $N^K$ possible, so lowering the cost to $\cal{O}(M)$ at the expense of some of the variance reduction (see~\cite{LinZCC:2005,lindsten2017divide} for similar ideas in the Monte Carlo literature).   In particular, by choosing the $M$ permutations so that the correlations among them are minimized (e.g.\ the $M$ permutations with least overlap among their components), it might be possible to substantially reduce the cost without sacrificing too much of the variance reduction. 
Indeed, by setting the number $M$ of permutations to be such that $M$ evaluations of the test function incurs a cost comparable to that of generating the $N$ unpermuted tuples, one can ensure that the overall cost of the resulting estimator never greatly exceeds that of the conventional estimator. 
This type of approach has been studied in the sparse grid literature~\cite{gerstner1998numerical},
and is closely related to the theory of incomplete U-statistics~\cite[Chap.\ 4.3]{Lee1990}, an area in which there are ongoing efforts directed at designing good reduced-cost estimators (e.g.\ \cite{kong2020design}).

There are, however, settings in which product-form estimators should be applied without hesitation: if the integrand is a sum of products (SOP) of univariate functions, the cost comes down to $\cal{O}(N)$ without affecting the variance reduction (Section~\ref{sec:lincost}).
For instance, when estimating ELBO gradients to optimize mean-field approximations~\cite{Ranganath2014} of posteriors $e^{v}$ with SOP potentials $v$. More generally, if the test function is a sum of partially-factorized functions, the estimators' evaluation costs can often be substantially reduced (see also Section~\ref{sec:lincost}) so that the variance reduction far outweighs the more mild increases in cost. For instance, as we saw with the applications of importance sampling and its product-form variant in Section~\ref{sec:numex}.

For integrands lacking this sort of structure, and   at the expense of introducing some bias,  these types of cost reductions can sometimes be retained if one is able to find a good SOP approximation to the integrand (Example~\ref{ex:taylor}). How to construct these approximations for generic functions (or for function classes of interest in given applications) is an open question upon whose resolution the success of this type of approach hinges.  In reality, combining product-form estimators with SOP approximations amounts to nothing more than an approximate dimensionality reduction technique: we approximate a high-dimensional integral with a linear combination of products of low-dimension integrals, estimate each of the latter separately, and plug the estimates back into the linear combination to obtain an estimate of the original integral. It is certainly not without precedents: for instance, \cite{Rahman2004,Ma2009,Gershman2012,Braun2010} all propose, in rather different contexts, similar approximations except that the low-dimensional   integrals are computed using closed-form expressions or quadrature (for a very well-known example, see the delta method for moments~\cite{Oehlert1992}). In practice, the best option will likely involve a mix of these: use closed-form expressions where available, quadrature where possible, and Monte Carlo (or Quasi Monte Carlo) for everything else.

About the computational resources required to evaluate product-form estimators, and the allocation thereof, we ought to mention one interesting variant of the estimators that we omitted from the main text to keep the exposition simple. 
In particular, throughout we assumed that the same number of samples  are drawn from each marginal $\mu_1,\dots,\mu_K$ of the product-form target or proposal $\mu$. This need not be the case: straightforward extensions of our arguments show that the estimator
$$\mu^{N_1,\dots,N_K}_\times(\varphi):=\frac{1}{\prod_{k=1}^K N_k}\sum_{n_1=1}^{N_1}\dots\sum_{n_K=1}^{N_K}\varphi(X_1^{n_1},\dots,X_K^{n_K})$$
behaves much as~\eqref{eq:prodestK} does, even if a different number of samples $N_k$ are used per marginal $\mu_k$. This variant potentially allows us to concentrate our computational budget on `the most important dimensions', an idea that has found significant success in other areas of numerical integration (e.g.\ \cite{gerstner1998numerical,Gerstner2003,owen1998monte}). In our case, this could be done using the pertinent generalizations of the variance expressions in Theorem~\ref{thrm:classicalK}, which are identical  except that $N^{\mmag{A}}$ therein must be replaced by $\prod_{k\in A}N_k$ (these can be obtained by retracing the steps in the theorem's proof). In particular, one could  estimate the terms in these expressions and adjust the sample sizes so that the estimator variance is minimized, potentially in an iterative manner leading to  an adaptive scheme.

Combining product-form estimators with other Monte Carlo methodology expands their utility beyond product-form targets. We illustrated this in Section~\ref{sec:extend} by describing the three simplest and most readily accessible such combinations we could think of: their merger with  importance sampling  applicable to targets that are absolutely continuous with respect to product-form distributions (Section~\ref{sec:is}), that with  importance sampling squared applicable to targets that are absolutely continuous with respect to partially-factorized distributions (Section~\ref{sec:ppf}, see also~\cite{Tran2013}), and that with  pseudo-marginal MCMC applicable to targets with intractable densities (Section~\ref{sec:pMH}, see also~\cite{Schmon2021}). In all of these cases, we demonstrated theoretically that the resulting estimators are more statistically efficient than their standard counterparts (Corollaries~\ref{cor:is}--\ref{cor:pmh}). Many other extensions are possible. For instance, one can embed product-form estimators within random weight particle filters~\cite{rousset2006,fearnhead2008,fearnhead2010random}---and, more generally, algorithms reliant on unbiased estimation---much the same way we did for IS$^2$ and GIMH in Sections~\ref{sec:ppf}--\ref{sec:pMH}. For an example of a slightly different vein,  see Appendix~\ref{sec:mop} where we consider `mixture-of-product-form' estimators applicable to targets which are mixtures of product-form distributions and, by combining these with importance sampling, we obtain a product-form version of (stratified) mixture importance sampling estimators~\cite{Oh1993,Hesterberg1995} that is particularly appropriate for multi-modal targets. For further examples, see the divide-and-conquer SMC algorithm~\cite{lindsten2017divide,Kuntz2021} obtained by combining product-form estimators with SMC and Tensor Monte Carlo~\cite{Aitchison2019} obtained by merging the estimators with variational autoencoders.

When choosing among the resulting (and at times bewildering) constellation of estimators, we recommend following one simple principle:  pick estimators that somehow `resemble' or `mirror' the target. Good examples of this are well-parametrized Gibbs samplers which generate new samples using the target's exact conditional distributions and, consequently, often outperform other Monte Carlo algorithms (e.g.\ Section~\ref{sec:numex}). While for many targets these conditional distributions cannot be obtained (nor are good parametrizations known), their (conditional) independence structure is usually obvious (e.g.\ see \cite{Gelman2006,Gelman2006a,Koller2009,hoffman2013stochastic,Blei2003} and the many references therein) and can be mirrored using product-form estimators within one's methodology of choice. 
Indeed, in the case of the simple hierarchical model (Example~\ref{ex:hier}), it was the PFIS$^2$ estimator utilizing samples with exactly the same independence structure as the model's that performed best (besides the Gibbs sampler). Of course, this model's independence structure was particularly simple, and so were the resulting estimators. However, we believe that broadly the same considerations apply to  models with more complex structures and that product-form estimators can be adapted to such structures by following analogous steps.

To summarize, we believe that product-form estimators are of greatest use not on their own, but  embedded within more complicated Monte Carlo routines to tackle the aspects of the problem exhibiting product structure. There remains much work to be done in this direction.

\bibliographystyle{plainurl} 
\bibliography{../prod_biblio}       
\appendix

\section{Proof of Theorem~\ref{thrm:classicalK}}\label{app:classicalproof}
We begin with the proof of Lemma~\ref{lem:iidl2}:
\begin{proof}[Proof of Lemma~\ref{lem:iidl2}] We start with a simple identity: because $\sum_{i=0}^j(-1)^i\binom{j}{i}$ vanishes unless $j=0$,
\begin{align}
\sum_{B\subseteq A}(-1)^{\mmag{B}}&=\sum_{i=0}^{\mmag{A}}\sum_{B\subseteq A:\mmag{B}=i}(-1)^{\mmag{B}}=\sum_{i=0}^{\mmag{A}}(-1)^i\sum_{B\subseteq A:\mmag{B}=i}1=\sum_{i=0}^{\mmag{A}}(-1)^i\binom{\mmag{A}}{i}=1_{\{\mmag{A}=0\}}\nonumber\\
&=1_{\{A=\emptyset\}}\quad\forall B\subseteq A\subseteq [K].\label{eq:identity}
\end{align}

Next, we show that, for any non-empty $A\subseteq[K]$ and $\psi$ belonging to $L^2_{\mu_A}$, $\psi_A$ in~\eqref{eq:proderr} satisfies
\begin{equation}\label{eq:special}
\mu_B(\psi_A)=0\quad\forall \emptyset\neq B\subseteq A.
\end{equation}
Because $\mu_B(\psi_A)=\mu_{B\backslash\{k\}}(\mu_k(\psi_A))$ for any $k$ in a non-empty subset $B$ of $A$, we need only show that $\mu_k(\psi_A)=0$ for any $k$ in $A$. Fix any such $k$ and note that
\begin{align*}(-1)^{\mmag{A}}\psi_A=&\sum_{B\subseteq A\backslash\{k\}}(-1)^{-\mmag{B}}\mu_{A\backslash B}(\psi)+\sum_{B\subseteq A\backslash \{k\}}(-1)^{-\mmag{B\cup\{k\}}}\mu_{A\backslash (B\cup\{k\})}(\psi)\\
=&\sum_{B\subseteq A\backslash\{k\}}(-1)^{-\mmag{B}}\mu_{A\backslash B}(\psi)-\sum_{B\subseteq A\backslash \{k\}}(-1)^{-\mmag{B}}\mu_{A\backslash (B\cup\{k\})}(\psi).\end{align*}
We obtain  $\mu_k(\psi_A)=0$ by integrating both sides with respect to $\mu_k$, and~\eqref{eq:special} follows. 

Returning to~\eqref{eq:proderr}, assume, without loss of generality, that $A=\{1,\dots,\mmag{A}\}$, and note that
\begin{align*}
\Ebb{\left[\left(\prod_{k\in A}[\mu^N_{k}-\mu_k]\right)(\psi)\right]^2}&=\Ebb{\mu_A^N(\psi_A)^2}=\Ebb{\left(\frac{1}{N^{\mmag{A}}}\sum_{\bm{i}\in [N]^{\mmag{A}}}\psi_A(X^{\bm{i}})\right)^2}\\
&=\frac{1}{N^{2\mmag{A}}}\sum_{\bm{i}\in [N]^{\mmag{A}}}\sum_{\bm{j}\in [N]^{\mmag{A}}}\Ebb{\psi_A(X^{\bm{i}})\psi_A(X^{\bm{j}})},
\end{align*}
where $X^{\bm{i}}:=(X^{i_k}_k)_{k=1}^{\mmag{A}}$ and similarly for $X^{\bm{j}}$. 
Suppose that $i_k\neq j_k$ for some $k$ and let $\cal{F}^{i_k,j_k}$ denote the sigma algebra generated by $(X^{i_{l}}_{l},X^{j_{l}}_{l})_{l\neq k}$. Because the components are independent and both $X^{i_k}_k$ and $X^{j_k}_k$ have law $\mu_k$, 
\begin{align*}
\Ebb{\psi_A(X^{\bm{i}})\psi_A(X^{\bm{j}})}&=\Ebb{\Ebb{\psi_A(X^{\bm{i}})\psi_A(X^{\bm{j}})\middle|\cal{F}^{i_k,j_k}}}=\Ebb{\Ebb{\psi_A(X^{\bm{i}})\middle|\cal{F}^{i_k,j_k}}\Ebb{\psi_A(X^{\bm{j}})\middle|\cal{F}^{i_k,j_k}}}\\
&=\Ebb{\mu_k(\psi_A)\mu_k(\psi_A)}=0,
\end{align*}
where the last equality follows from~\eqref{eq:special}. Hence,
\begin{align}
\Ebb{\left[\left(\prod_{k\in A}[\mu^N_{k}-\mu_k]\right)(\psi)\right]^2}&=\frac{1}{N^{2\mmag{A}}}\sum_{\bm{i}\in[N]^{\mmag{A}}}\Ebb{\psi_A(X^{\bm{i}})^2}=\frac{1}{N^{2\mmag{A}}}\sum_{\bm{i}\in[N]^{\mmag{A}}}\mu_A(\psi_A^2)=\frac{\mu_A(\psi_A^2)}{N^{\mmag{A}}}.\label{eq:upsa2}
\end{align}
Because $A$ is non-empty by assumption and $(-1)^{\mmag{B}}=(-1)^{-\mmag{B}}$ for all $B\subseteq A$,~\eqref{eq:identity} implies that 
\begin{align*}
    \mu_A(\psi_A^2) = & \mu_A\left(\left[\psi_A-\mu_A(\psi)(-1)^{\mmag{A}}\sum_{B\subseteq A}(-1)^{-\mmag{B}}\right]^2\right)= \mu_A\left(\left[\sum_{B\subseteq A}(-1)^{\mmag{A}-\mmag{B}}[\mu_{A\backslash B}(\psi)-\mu_A(\psi)]\right]^2\right)\\
    =&\sum_{B\subseteq A}\sum_{B'\subseteq A}(-1)^{\mmag{B}+\mmag{B'}}\mu_A([\mu_{A\backslash B}(\psi)-\mu_A(\psi)][\mu_{A\backslash B'}(\psi)-\mu_A(\psi)]).
\end{align*} 
But,
\begin{align*}
&\mu_A([\mu_{A\backslash B}(\psi)-\mu_A(\psi)][\mu_{A\backslash B'}(\psi)-\mu_A(\psi)])\\
&\qquad\qquad\qquad\qquad\qquad\qquad=\mu_{B\cap B'}(\mu_{B\backslash B'}([\mu_{A\backslash B}(\psi)-\mu_A(\psi)])\mu_{B'\backslash B}([\mu_{A\backslash B'}(\psi)-\mu_A(\psi)]))\\%
&\qquad\qquad\qquad\qquad\qquad\qquad=\mu_{B\cap B'}([\mu_{A\backslash (B\cap B')}(\psi)-\mu_A(\psi)]^2),
\end{align*}
whence,
\begin{align}\mu_A(\psi_A^2)=&\sum_{B\subseteq A}\sum_{B'\subseteq A}(-1)^{\mmag{B}+\mmag{B}'}\mu_{B\cap B'}([\mu_{A\backslash (B\cap B')}(\psi)-\mu_A(\psi)]^2)\nonumber\\
=&\sum_{C\subseteq A}\mu_{C}( [\mu_{A\backslash C}(\varphi)-\mu_A(\varphi)]^2)\left(\sum_{B,B'\subseteq A:B\cap B'=C}(-1)^{\mmag{B}+\mmag{B'}}\right)\label{eq:dmwua9dnwaud}.\end{align}
However, our starting identity implies that
\begin{align*}
\sum_{B,B'\subseteq A:B\cap B'=C}(-1)^{\mmag{B}+\mmag{B'}}&=\sum_{D\subseteq{A\backslash C}}\sum_{D'\in \subseteq{A\backslash (C\cup D)}}(-1)^{\mmag{C\cup D}+\mmag{C\cup D'}}=\sum_{D\subseteq {A\backslash C}}\sum_{D'\subseteq{A\backslash (C\cup D)}}(-1)^{\mmag{D}+\mmag{D'}}\\
&=\sum_{D\subseteq {A\backslash C}}(-1)^{\mmag{D}}\left(\sum_{D'\subseteq{A\backslash (C\cup D)}}(-1)^{\mmag{D'}}\right)=\sum_{D\subseteq {A\backslash C}}(-1)^{\mmag{D}}1_{\{A\backslash (C\cup D)=\emptyset\}}\\
&=(-1)^{\mmag{A\backslash C}}=(-1)^{\mmag{A}-\mmag{C}},
\end{align*}
and the lemma follows from~\eqref{eq:dmwua9dnwaud}.
\end{proof}
We are now ready to tackle the proof of Theorem~\ref{thrm:classicalK}. We do it in steps:
\begin{proof}[Step 1: unbiasedness]\let\qed\relax Because $X^{\bm{n}}\sim \mu$ for all $\bm{n}$ in $[N]^K$, this follows directly from the estimator's definition in~\eqref{eq:prodestK}.\end{proof}
\begin{proof}[Step 2: square integrability of $\mu_{A^c}(\varphi)$]\let\qed\relax That, for any subset $A$ of $[K]$, $\mu_{A^c}(\varphi)$ is square $\mu_A$-integrable follows from our assumption that $\varphi$ in square $\mu$-integrable and Jensen's inequality:
$$\mu_A(\mu_{A^c}(\varphi)^2)\leq \mu_A(\mu_{A^c}(\varphi^2))=\mu(\varphi^2)<\infty\quad\forall A\subseteq [K].$$
\end{proof}
\begin{proof}[Step 3: variance expressions]\let\qed\relax By~(\ref{eq:decomp},\ref{eq:proderr}),
\begin{align}
\text{Var}(\mu^N_\times(\varphi))&=\Ebb{[\mu^N_\times(\varphi)-\mu(\varphi)]^2}=\Ebb{\left[\sum_{\emptyset\neq A\subseteq[K]}\left(\prod_{k\in A}(\mu^N_k-\mu_k)\right)(\mu_{A^c}(\varphi))\right]^2}\nonumber\\
&=\Ebb{\left[\sum_{\emptyset\neq A\subseteq[K]}\mu^N_A(\psi_A)\right]^2}
=\sum_{\emptyset\neq A\subseteq[K]}\sum_{\emptyset\neq A'\subseteq[K]}\Ebb{\mu^N_A(\psi_A)\mu^N_{A'}(\psi_{A'})},\label{eq:fmewa9fhwfbwaf}
\end{align}
where $\psi_A$ is as in~\eqref{eq:proderr} with $\psi:=\mu_{A^c}(\varphi)$. Fix any non-empty subsets $A,A'$ of $[K]$ such that $A\backslash A'$ is non-empty, set $\cal{F}_{A'}$ to be the sigma algebra generated by all $(X_{l}^n)_{n=1}^N$ with $l$ in $A'$, and note that
\begin{align*}
\Ebb{\mu^N_A(\psi_A)\mu^N_{A'}(\psi_{A'})}&=\Ebb{\Ebb{\mu^N_A(\psi_A)\mu^N_{A'}(\psi_{A'})|\cal{F}_{A'}}}=\Ebb{\Ebb{\mu^N_A(\psi_A)|\cal{F}_{A'}}\mu^N_{A'}(\psi_{A'})}.
\end{align*}
But, once again assuming without loss of generality that $A=\{1,\dots,\mmag{A}\}$ and applying~\eqref{eq:special},
\begin{align*}
\Ebb{\mu^N_A(\psi_A)|\cal{F}_{A'}}&=\frac{1}{N^{\mmag{A}}}\sum_{\bm{i}\in[N]^{\mmag{A}}}\Ebb{\psi_A(X^{\bm{i}})|\cal{F}_{A'}}=\frac{1}{N^{\mmag{A}}}\sum_{\bm{i}\in[N]^{\mmag{A}}}\mu_{A\backslash A'}(\psi_A)((X_{l}^{i_l})_{l\in A'}))=0
\end{align*}
almost surely, and we find that $\Ebb{\mu^N_A(\psi_A)\mu^N_{A'}(\psi_{A'})}=0$. Because an analogous argument shows that this is the case if $A'\backslash A$ is non-empty,~\eqref{eq:fmewa9fhwfbwaf} reduces to
\begin{align*}\text{Var}(\mu^N_\times(\varphi))&=\sum_{\emptyset\neq A\subseteq[K]}\Ebb{\mu^N_A(\psi_A)^2}=\sum_{\emptyset\neq A\subseteq[K]}\Ebb{\left[\left(\prod_{k\in A}[\mu^N_{k}-\mu_k]\right)(\mu_{A^c}(\varphi))\right]^2},\end{align*}
and the variance expressions in~\eqref{eq:vark} follow from Lemma~\ref{lem:iidl2}.
\end{proof}

\begin{proof}[Step 4: consistency and asymptotic normality]\let\qed\relax From~\eqref{eq:decomp}, it follows that, for all $N>0$,
\begin{equation}\label{eq:decomp3}\mu^N_\times(\varphi)-\mu(\varphi)=\sum_{k=1}^K[\mu^N_k(\mu_{\{k\}^c}(\varphi))-\mu(\varphi)]+R_N\quad \forall N>0,\end{equation}
where $R_N:=\sum_{A\subseteq [K]:\mmag{A}>1}\left(\prod_{k\in A}[\mu^N_{k}-\mu_k]\right)(\mu_{A^c}(\varphi))$. Given the  square integrability in Step 2 and that $\mu(\varphi)=\mu_k(\mu_{\{k\}^c}(\varphi))$,   the classical law of large numbers and central limit theorem   imply that
\begin{align*}
&\lim_{N\to\infty}\mu^N_k(\mu_{\{k\}^c}(\varphi))-\mu(\varphi)=0\enskip\text{almost surely},\\
& N^{1/2}[\mu^N_k(\mu_{\{k\}^c}(\varphi))-\mu(\varphi)]\Rightarrow\cal{N}(0,\sigma_{k}^2(\varphi))\enskip\text{as}\enskip N\to\infty,
\end{align*}
for all $k$ in $[K]$. Because $(X_1^n)_{n=1}^\infty$, $\dots$, $(X_K^n)_{n=1}^\infty$ are independent sequences, so are 
$$(\mu_1^N(\mu_{\{1\}^c}(\varphi))-\mu(\varphi))_{N=1}^\infty,\, \dots,\, (\mu_K^N(\mu_{\{K\}^c}(\varphi))-\mu(\varphi))_{N=1}^\infty,$$
and the continuous mapping theorem yields
\begin{align*}&\sum_{k=1}^K[\mu_k^N(\mu_{\{k\}^c}(\varphi))-\mu(\varphi)]\to 0\enskip\text{almost surely},\\
& \sum_{k=1}^K N^{1/2}[\mu_k^N(\mu_{\{k\}^c}(\varphi))-\mu(\varphi)]\Rightarrow \cal{N}(0,\sigma_{\times}^2(\varphi)),\end{align*}
as $N\to\infty$. Hence, if we can show that the remainder term $R_N$ tends to zero fast enough, i.e.\
\begin{align}
R_N\to0\enskip\text{almost surely},\quad N^{1/2}R_N\Rightarrow0,\quad\text{as }N\to\infty,\label{eq:fast2}
\end{align}
then (\ref{eq:lln},\ref{eq:clt}) follow from~\eqref{eq:decomp3} and Slutsky's theorem. To do so, note that the same argument as in Step 3 shows that
\begin{align*}
\Ebb{R_N^2}&=\Ebb{\left[\sum_{ A\subseteq[K]:\mmag{A}>1}\mu^N_A(\psi_A)\right]^2}=\sum_{A\subseteq [K]:\mmag{A}>1}\frac{1}{N^{\mmag{A}}}\sum_{B\subseteq A}(-1)^{\mmag{A}-\mmag{B}}\sigma_{A,B}^2(\mu_{A^c}(\varphi))\quad\forall N>0.
\end{align*}
Markov's inequality then implies that
\begin{align*}&\Pbb{\{|R_N|\geq \varepsilon\}}\leq \frac{\Ebb{R_N^2}}{\varepsilon^2}=\sum_{A\subseteq [K]:\mmag{A}>1}\frac{1}{N^{\mmag{A}}\varepsilon^2}\sum_{B\subseteq A}(-1)^{\mmag{A}-\mmag{B}}\sigma_{A,B}^2(\mu_{A^c}(\varphi))\quad\forall \varepsilon>0,\enskip N>0.\end{align*}
Summing over $N$ and applying a Borel-Cantelli argument, we obtain the first limit in \eqref{eq:fast2}. For the second, set $\varepsilon:=N^{-1/2}\delta$, for any $\delta>0$, and take the limit $N\to\infty$.
\end{proof}
\section{Proofs of Theorem~\ref{thrm:mvu} and its corollary}\label{app:mvu}
Throughout this appendix we use colons to indicate ranges, with $p:q = p,p+1,\ldots,q$ for integers $p \leq q$, and use ranges in indices to compactly describe tuples. For instance, $x^{1:N}$ or $x_{1:K}^{1:N}$ refer to a tuple in $S^N$, $x_k^{1:N}$ to one in $S_k^N$, $x_{1:K}^n$ to one in $S$, etc.
The key to Theorem's~\ref{thrm:mvu} proof is the following.
\begin{lemma}\label{lem:sym}
Suppose that $f$ is a measurable real-valued  function on $(S^N,\cal{S}^N)$ satisfying
$$\left(\prod_{n=1}^N \mu\right)(f)=0$$
for all product-form $\mu$ on $(S,\cal{S})=(S_1\times\dots \times S_K,\cal{S}_1\times\dots\times\cal{S}_K)$ satisfying $\mu(\mmag{\varphi})<\infty$ for some given measurable $\varphi:S\to\r$. It is symmetric---meaning that $f$ remains unchanged by permutations of its arguments: 
\begin{align}&f(x_1^{\tau_1(1)},\dots,x_1^{\tau_1(N)},\dots,x_K^{\tau_K(1)},\dots,x_K^{\tau_K(N)})=f(x_1^{1},\dots,x_1^{N},\dots,x_K^{1},\dots,x_K^{N})\label{eq:fneya8fney8a},\end{align}
\sloppy for all $x_1^{1:N}$ in $S_1^N$, $\dots$, $x_K^{1:N}$  in $S_K^N$, and permutations $\tau_1,\dots,\tau_K$  of $[N]$---if and only if $f$ is zero everywhere. 
\end{lemma}

\begin{proof}The reverse direction is trivial. For the forward one, pick any integer $N$,  $x_{1}^{1:N}$ in $S_1^N,$ $\dots,$ $x_{K}^{1:N}$ in $S_K^N$, $K$ probability distributions $w_{1}^{1:N},\dots,w_{K}^{1:N}$ on $[N]$,  and let
$$\mu:=\prod_{k=1}^K \sum_{l_k^n=1}^N w_{k}^{l_k}\delta_{x_{k}^{l_k}}.$$
By the premise, 
\begin{align}
0=&\left(\prod_{n=1}^N \mu\right)(f)=\left(\prod_{n=1}^N\prod_{k=1}^K \sum_{l_k^n=1}^N w_{k}^{l_k^n}\delta_{x_{k}^{l_k^n}}\right)(f)=\left(\prod_{k=1}^K\prod_{n=1}^N \sum_{l_{k}^n=1}^N w_{k}^{l_k^n}\delta_{x_{k}^{l_k^n}}\right)(f)\nonumber\\
=&\left(\prod_{k=1}^K\sum_{l_k^{1:N}\in[N]^N}w_{k}^{l_k^1}\dots w_{k}^{l_k^N}\delta_{x_k^{l_k^1}}\times\dots\times \delta_{x_k^{l_k^N}}\right)(f)\nonumber\\
=&\sum_{l_{1:K}^{1:N}\in[N]^{KN}}w_{1}^{l_1^1}\dots w_{1}^{l_1^N}\dots w_{K}^{l_K^1}\dots w_{K}^{l_K^N} f(x^{l^1_1}_1,\dots,x^{l^N_1}_1,\dots,x^{l^1_K}_K,\dots,x^{l^N_K}_K)=:g(w_{1:K}^{1:N}).\label{eq:mge7r8hagyega}
\end{align}
The above holds for all $w_{1}^{1:N},\dots,w_{K}^{1:N}$ in the unit simplex and it follows that $g$ is zero on the positive orthant $\r^{NK}_{>0}$: for any  $z_{1:K}^{1:N}$ in $\r^{NK}_{>0}$, set 
$$w_k^{1:N}:=\frac{z_k^{1:N}}{s_k}\quad\text{where} \quad s_k:=\sum_{l=1}^Nz_k^l,\quad \forall k\in [K],$$
to find that
\begin{align*}
g(z_{1:K}^{1:N})&=g(s_1w_1^{1:N},\dots, s_Kw_K^{1:N})=(s_1\dots s_K)^Ng(w_{1:K}^{1:N})=0.
\end{align*}
Given that $g$ is a polynomial, it must be zero everywhere (to argue this, use induction and the fact that a non-zero univariate polynomial of degree $d$ has at most $d$ roots). In other words, all of $g$'s coefficients are zero. 

By definition, $g$ exclusively features monomials of the type $[w_1^1]^{d_1^1}\dots[w_1^N]^{d_1^N}\dots[w_K^1]^{d_K^1}\dots[w_K^N]^{d_K^N}$ with exponents satisfying 
\begin{equation}\label{eq:fne87nfea78hfanwew}\sum_{l=1}^Nd_k^l=N\quad \forall k\in[K]\end{equation} For any such exponents, the monomial's coefficient $c$ is given by
$$c=\sum_{l^{1:N}_{1:K}}f(x^{l^1_1}_1,\dots,x^{l^N_1}_1,\dots,x^{l^1_K}_K,\dots,x^{l^N_K}_K),$$
where the sum is taken over all indices $l_1^1,\dots,l_1^N,\dots,l_K^1,\dots,l_K^N$ in $[N]$ satisfying
\begin{equation}\label{eq:fne87nfea78hfanwew2}
\mmag{\{n\in[N]:l_k^n=l\}}=d_k^l\quad\forall l\in[N],\enskip k\in[K].
\end{equation}
Because we obtain all such index vectors by picking any one of them, say $(l_1^1,\dots,l_1^N,\dots,l_K^1,\dots,l_K^N)$, and permuting its entries, we find that the coefficient $c$ equals
$$\sum_{\tau_1,\dots,\tau_K}f(x^{l^{\tau_1(1)}_1}_1,\dots,x^{l^{\tau_1(N)}_1}_1,\dots,x^{l^{\tau_K(1)}_K}_K,\dots,x^{l^{\tau_K(N)}_K}_K),$$
where the sum is taken over all permutations $\tau_1,\dots,\tau_K$ of $[N]$. Given that $c=0$, combining the above with \eqref{eq:fneya8fney8a} we find that 
$$f(x^{l^1_1}_1,\dots,x^{l^N_1}_1,\dots,x^{l^1_K}_K,\dots,x^{l^N_K}_K)=0.$$
Because~(\ref{eq:fne87nfea78hfanwew},\ref{eq:fne87nfea78hfanwew2}) hold if $l_k^n:=n$ for all $n$ in $[N]$ and $k$ in $[K]$, the above implies that
$$f(x^{1}_1,\dots,x^{N}_1,\dots,x^{1}_K,\dots,x^{N}_K)=0.$$
and the lemma follows because $x^{1:N}_{1:K}$ was arbitrary.
\end{proof}
Theorem~\ref{thrm:mvu} now follows easily:
\begin{proof}[Proof of  Theorem~\ref{thrm:mvu}]Let
\begin{align*}&g(x_{1:K}^{1:N}):=\frac{1}{(N!)^K}\sum_{\tau_1,\dots,\tau_K}f(x_1^{\tau_1(1)},\dots,x_1^{\tau_1(N)},\dots,x_K^{\tau_K(1)},\dots,x_K^{\tau_K(N)})\end{align*}
for all $x_1^{1:N}$ in $S_1^N$, $\dots$, $x_K^{1:N}$  in $S_K^N$, where the sum is over all permutations $\tau_1,\dots,\tau_K$  of $[N]$. Clearly, $g$ is symmetric (in the sense of Lemma~\ref{lem:sym}) and 
\begin{align*}\left(\prod_{n=1}^N \mu\right)(g)&=\frac{1}{(N!)^K}\sum_{\tau_1,\dots,\tau_K}\left(\prod_{n=1}^N \mu\right)(f)=\frac{1}{(N!)^K}\sum_{\tau_1,\dots,\tau_K}\mu(\varphi)=\frac{(N!)^K}{(N!)^K}\mu(\varphi)=\mu(\varphi)\end{align*}
for all product-form $\mu$ on $(S,\cal{S})$. Because, furthermore, $\mu^N_\times(\varphi)=h(X^{1:N}_{1:K})$ where  $
h(x^{1:N}_{1:K}):=N^{-K}\sum_{\bm{n}\in[N]^K}\varphi(x^{\bm{n}})$ 
is symmetric and satisfies $\left(\prod_{n=1}^N\mu\right)(h)=\Ebb{\mu^N_\times(\varphi)}=\mu(\varphi)$,  Lemma~\ref{lem:sym} implies that $g(X^{1:N}_{1:K})=\mu^N_\times(\varphi)$. The theorem then follows from the Cauchy-Schwarz inequality and the fact $X^{1:N}_{1:K}$'s law remains unchanged by component-wise swaps:
\begin{align*}
&(N!)^{2K}\text{Var}(\mu_\times^N(\varphi))=(N!)^{2K}\text{Var}(g(X^{1:N}_{1:K}))\\
&=\text{Var}\Bigg(\sum_{\tau_1,\dots,\tau_K}f(X_1^{\tau_1(1)},\dots,X_1^{\tau_1(N)}, \dots,X_K^{\tau_K(1)},\dots,X_K^{\tau_K(N)})\Bigg)\\
&=\sum_{\tau_1,\dots,\tau_K}\sum_{\tau_1',\dots,\tau_K'}\text{Cov}( f(X_1^{\tau_1(1)},\dots,X_1^{\tau_1(N)},\dots,X_K^{\tau_K(1)},\dots,X_K^{\tau_K(N)}),\\
&\qquad\qquad\qquad\qquad  f(X_1^{\tau_1'(1)},\dots,X_1^{\tau_1'(N)},\dots,X_K^{\tau_K'(1)},\dots,X_K^{\tau_K'(N)}))\\
&\leq\sum_{\tau_1,\dots,\tau_K}\sum_{\tau_1',\dots,\tau_K'}\text{Var}(f(X_1^{\tau_1(1)},\dots,X_1^{\tau_1(N)},\dots,X_K^{\tau_K(1)},\dots,X_K^{\tau_K(N)}))^{1/2}\\
&\qquad\qquad\quad \times\text{Var}(f(X_1^{\tau_1'(1)},\dots,X_1^{\tau_1'(N)},\dots,X_K^{\tau_K'(1)},\dots,X_K^{\tau_K'(N)}))^{1/2}\\
&\leq \sum_{\tau_1,\dots,\tau_K}\sum_{\tau_1',\dots,\tau_K'}\text{Var}(f(X^{1:N}_{1:K}))=(N!)^{2K}\text{Var}(f(X^{1:N}_{1:K})).
\end{align*}
\end{proof}
Corollary~\ref{cor:varbound} follows almost immediately from Theorem~\ref{thrm:mvu} and Jensen's inequality:
\begin{proof}[Proof of Corollary~\ref{cor:varbound}]
Given that $\mu^N(\varphi)$ is an unbiased estimator of $\mu(\varphi)$ for all distributions $\mu$ satisfying $\mu(\mmag{\varphi})<\infty$, the first bound is a direct consequence of Theorem~\ref{thrm:mvu}. For the asymptotic variance bound, note that
\begin{align*}
\sigma_{1}^2(\varphi)&=\mu_1(\mu_{\{1\}^c}(\varphi-\mu_1(\varphi))^2)=\mu_1(\mu_2(\psi-\mu_1(\psi))^2),\\
\sigma_{2}^2(\varphi)&=\mu_2(\mu_{\{2\}^c}(\varphi-\mu_2(\varphi))^2)=\mu_2(\mu_1(\psi-\mu_2(\psi))^2).
\end{align*}
where $\psi:=\mu_{[3:K]}(\varphi)$ and $[3:K]:=\{3,\dots,K\}$. Because Jensen's inequality implies that $$\mu_2(\mu_1(\psi-\mu_2(\psi))^2)\leq \mu_2(\mu_1([\psi-\mu_2(\psi)]^2)),$$it follows that
\begin{align}
\sigma_{1}^2(\varphi)+\sigma_{2}^2(\varphi)
&\leq \mu_1(\mu_2(\psi-\mu_1(\psi))^2)+\mu_2(\mu_1([\psi-\mu_2(\psi)]^2))\nonumber\\
&=\mu_1([\mu_2(\psi)-\mu_{[2]}(\psi)]^2)+\mu_1(\mu_2([\psi-\mu_2(\psi)]^2))\nonumber\\
&=\mu_1(\mu_2(\psi)^2)-\mu_{[2]}(\psi)^2+\mu_1(\mu_2(\psi^2)-\mu_2(\psi)^2)\nonumber\\
&=\mu_{[2]}(\psi^2)-\mu_{[2]}(\psi)^2=\mu_{[2]}([\psi-\mu_{[2]}(\psi)]^2).\nonumber
\end{align}
Setting now $\psi:=\mu_{[4:K]}(\varphi)$, we have that
\begin{align*}\sigma_{1}^2(\varphi)+\sigma_{2}^2(\varphi)\leq\mu_{[2]}(\mu_3(\psi-\mu_{[2]}(\psi))^2),\quad \sigma_{3}^2(\varphi)=\mu_{3}(\mu_{[2]}(\psi-\mu_3(\psi))^2).\end{align*}
Hence, applying Jensen's inequality again, we obtain
$$\sigma_{1}^2(\varphi)+\sigma_{2}^2(\varphi)+\sigma_{3}^2(\varphi)\leq \mu_{[3]}(\mu_{[4:K]}(\varphi-\mu_{[3]}(\varphi))^2).$$
Iterating this argument gives us the desired bound:
\begin{align*}
\sigma^2_\times(\varphi)=\sum_{k=1}^K\sigma_{k}^2(\varphi)\leq& \mu_{[K]}(\mu_{\emptyset}(\varphi-\mu_{[K]}(\varphi))^2)=\mu([\varphi-\mu(\varphi)]^2)=\sigma^2(\varphi).
\end{align*}
\end{proof}
\section{Proof of Theorem~\ref{thrm:classicalK2} and its corollary}\label{app:classicalK2}
Theorem~\ref{thrm:classicalK2}'s proof is straightforward:
\begin{proof}[Proof of Theorem~\ref{thrm:classicalK2}]
By definition, $\mu^{M,N}_\times(\varphi):=M^{-1}\sum_{m=1}^MZ_m^N(\varphi)$, where $Z_m^N:=N^{-K}\sum_{\bm{n}\in[N]^K}\varphi(\theta^m,X^{m,\bm{n}})$, and the unbiasedness follows:
\begin{align*}
&\Ebb{Z_m^N|\theta^m}=\frac{1}{N^K}\sum_{\bm{n}\in[N]^K}\Ebb{\varphi(\theta^m,X^{m,\bm{n}})|\theta^m}=\frac{1}{N^K}\sum_{\bm{n}\in[N]^K}(\cal{M}\varphi)(\theta^m)=(\cal{M}\varphi)(\theta^m)\quad \text{a.s.},\\
&\Rightarrow\Ebb{Z_m^N}=\Ebb{(\cal{M}\varphi)(\theta^m)}=\mu_0(\cal{M}\varphi)=\mu(\varphi).
\end{align*}
Because $Z_m^N$ is a function of $(\theta^m,X^{m,1},\dots,X^{m,N})$ that does not depend on $m$ and the sequence $(\theta^m,X^{m,1},\dots,X^{m,N})_{m=1}^M$ is i.i.d. and drawn from $\mu_0\otimes\cal{M}^N$, $(Z_m^N)_{m=1}^M$ forms an i.i.d.\ sequence and the consistency follows from the law of large numbers:
\begin{align*}
\lim_{M\to\infty}\mu^{M,N}_\times(\varphi)=\lim_{M\to\infty}\frac{1}{M}\sum_{m=1}^MZ_m^N=\Ebb{Z_m^N}=\mu(\varphi)\quad\textrm{almost surely}.
\end{align*}
Similarly, because Jensen's inequality implies that
$$\Ebb{(Z_m^N)^2}\leq \frac{1}{N^K}\Ebb{\sum_{\bm{n}\in[N]^K}\varphi(\theta^m,X^{m,\bm{n}})^2}=\mu(\varphi^2),$$
the central limit theorem for sums of i.i.d.\ square-integrable random variables tells us that, if $\varphi$ belongs to $L^2_\mu$, then~\eqref{eq:ppfclt} holds with $\sigma^2_{\times,N}(\varphi)=\textrm{Var}(Z_1^N)$. The expression for the asymptotic (in $M$) variance then follows from the law of total variance:
\begin{align*}
\textrm{Var}(Z_m^N)&=\Ebb{\text{Var}(Z_m^N|\theta^m)}+\text{Var}(\Ebb{Z_m^N|\theta^m})=\Ebb{V_{x,\times}^N(\theta^m)}+\text{Var}((\cal{M}\varphi)(\theta^m))\\
&=\mu_0(V_{x,\times}^N)+\mu_0([\cal{M}\varphi-\mu(\varphi)]^2)\quad\forall m\in [M],
\end{align*}
where, for each $\theta$ in $\Theta$, $V_{x,\times}^N(\theta)$ denotes   the RHS of~\eqref{eq:vark} with  $\cal{M}(\theta,dx)$ replacing $\mu(dx)$ therein. Lastly, using once again the independence of $Z_1^{N},\dots,Z_M^{N}$, we obtain the variance expressions:
\begin{align}\label{eq:mfeu8ahfdsadsaby8abegfyag}
\text{Var}(\mu_{\times}^{M,N}(\varphi))=\frac{1}{M^2}\sum_{m=1}^M\textrm{Var}(Z_m^N)=\frac{\mu_0(V_{x,\times}^N)+\mu_0([\cal{M}\varphi)-\mu(\varphi)]^2)}{M}.
\end{align}
\end{proof}
Corollary~\ref{cor:ppf} follows easily from Theorem~\ref{thrm:classicalK2} and our variance bounds for product-form estimators:
\begin{proof}[Proof of Corollary~\ref{cor:ppf}]Note that $\mu^{M, N}(\varphi)$ is a particularly simple partially product-form estimator. So Theorem~\ref{thrm:classicalK2} tells us that it is asymptotically (in $M$) normal with asymptotic variance $\sigma^2_N(\varphi):=\mu_0([\cal{M}\varphi-\mu(\varphi)]^2)+N^{-1}\mu_0(\cal{M}[\varphi-\cal{M}\varphi]^2)$ and that its variance equals $\sigma^2_{N}(\varphi)/M$. Given that $\text{Var}(\mu^{N,M}_\times(\varphi))=\sigma^2_{\times,N}(\varphi)/M$, we need only show that $\sigma^2_{\times,N}(\varphi)\leq \sigma^2_{N}(\varphi)$. However, by~\eqref{eq:mfeu8ahfdsadsaby8abegfyag},
\begin{align*}\sigma^2_{\times,N}(\varphi)- \sigma^2_{N}(\varphi)=\mu_0(V_{x,\times}^N)-\frac{\mu_0(\cal{M}[\varphi-\cal{M}\varphi]^2)}{N}=\mu_0\left(V_{x,\times}^N-\frac{\cal{M}[\varphi-\cal{M}\varphi]^2}{N}\right),\end{align*}
Theorem~\ref{thrm:classicalK} and Corollary~\ref{cor:varbound} imply that $V_{x,\times}^N$ is (pointwise) bounded above by $\cal{M}[\varphi-\cal{M}\varphi]^2$, and the result follows.
\end{proof}
\section{Proof of Theorem~\ref{thrm:mvuppf}}\label{app:ppfmvu}
In this appendix we use colon notation analogous to that introduced in Appendix~\ref{app:mvu}. This proof is a conceptually straightforward (although notationally arduous) extension of that given for Theorem~\ref{thrm:mvu} in Appendix~\ref{app:mvu}, and we only sketch it. In short, one needs to exploit the two symmetries present in the problem---i.e.\
\begin{enumerate}
\item the law of $(\theta^1,X^{1,1:N}),\dots,(\theta^M,X^{M,1:N})$ is invariant to permutations of the labels $1,...,M$,
\item for each $m$, the law of $X^{m,1},\dots,X^{m,N}$ is invariant to component-wise swaps of the entries of these tuples,
\end{enumerate}
---in much the same way as the second of these symmetries was exploited in Appendix~\ref{app:mvu}.
First off, Lemma~\ref{lem:sym} generalizes as follows.
\begin{lemma}\label{lem:symppf}
Suppose that $f$ is a measurable real-valued  function on $(\Theta\times S^N)^M$ satisfying
$$(\mu_0\otimes\cal{M}^N)^M(f)=0$$
for all partially factorized $\mu=\mu_0\otimes \cal{M}$ on $\Theta\times S$ satisfying~\eqref{eq:continuous} and $\mu(\mmag{\varphi})<\infty$ for some given measurable $\varphi:\Theta\times S\to\r$. It is symmetric in the sense of 1-2 above if and only if $f(\theta^{1:M},x^{1:M,1:N}_{1:K})=0$ for all $\theta^{1:M}$ in $\Theta^{M}$ and $x^{1:M,1:N}_{1:K}$ in $S^{NM}$ such that $\theta^i\neq\theta^j$ for all $i\neq j$.
\end{lemma}

\begin{proof}Similar to the proof of Lemma~\ref{lem:sym}: because, as explained in Remark~\ref{rem:ppfcaveat},~\eqref{eq:continuous} implies that no $\theta^m$s coincide, the reverse direction is immediate. For the forward direction,  we need to show that, for any given $\theta^{1:M}$ in $\Theta^M$ and $x^{1:M,1:N}_{1:K}$ in $S^{MN}$ with $\theta^i\neq\theta^j$ for all $i\neq j$,
\begin{equation}\label{eq:fwmua9ifnwuanaf}
f(\theta^{1:M},x^{1:M,1:N}_{1:K})=0.
\end{equation}
To do so, pick any $w_0^{1:M}$ in $\cal{P}([M])$ and $w_1^{1,1:N},\dots,w_1^{M,1:N},\dots,w_K^{1,1:N},\dots,w_K^{M,1:N}$ in $\cal{P}([N])$, where $\cal{P}([M])$ and $\cal{P}([N])$ denote $(M-1)$- and $(N-1)$-dimensional probability simplexes (with associated Borel sigma-algebras).  Because $\cal{T}$ contains all singletons, it is straightforward to check that the functions   $x_{1}^{1:N}:\Theta\to S_1^N$, $\dots$, $x_{K}^{1:N}:\Theta\to S_K^N$ and $w_{1}^{1:N}:\Theta\to\cal{P}([N])$, $\dots$, $w_{K}^{1:N}:\Theta\to\cal{P}([N])$ defined by
$$x_k^{1:N}(\theta):=\sum_{r=1}^Mx_k^{r,1:N}1_{\{\theta^r\}}(\theta),\quad w_k^{1:N}(\theta)=\sum_{r=1}^Mw_k^{r,1:N}1_{\{\theta^r\}}(\theta),\quad\forall \theta\in\Theta,\enskip k\in[K],$$
are measurable. Because we have chosen $\theta^{1:M}$ so that none coincide, the functions satisfy%
$$x_k^{1:N}(\theta^r)=x_k^{r,1:N},\quad w_k^{1:N}(\theta^r)=w_k^{r,1:N},\quad \forall r\in[M],\enskip k\in [K].$$
Next, define
$$\mu_0(d\theta):=\sum_{r=1}^M w_0^{r}\delta_{\theta^{r}}(d\theta),\qquad \cal{M}_k(\theta,dx_k):=\sum_{l_k=1}^N w_{k}^{l_k}(\theta)\delta_{x_{k}^{l_k}(\theta)}(dx_k)\quad\forall k\in[K];$$
in which case~\eqref{eq:mge7r8hagyega} becomes 
\begin{align*}
0=&(\mu_0\otimes \cal{M}^N)^M(f)=\left(\mu_0\otimes\left(\prod_{n=1}^N\prod_{k=1}^K \sum_{l_k^n=1}^N w_{k}^{l_k^n}(\theta)\delta_{x_{k}^{l_k^n}(\theta)}\right)\right)^M(f)\\
=&\dots=\left(\mu_0\otimes\left(\sum_{l_{1:K}^{1:N}\in[N]^{KN}}\prod_{n=1}^N\prod_{k=1}^K w_{k}^{l_k^n}(\theta)\delta_{x^{l^n_k}_k(\theta)}\right)\right)^M(f)\\
=&\left(\sum_{r=1}^M\sum_{l_{1:K}^{1:N}\in[N]^{KN}}w_0^r\delta_{\theta^r}\times\prod_{n=1}^N\prod_{k=1}^K w_{k}^{r,l_k^n}\delta_{x^{r,l^n_k}_k}\right)^M(f)\\
=&\sum_{r_{1:M}\in[M]^M}\sum_{l_{1:K}^{1:N}\in[N]^{KN}}\Bigg[\left(\prod_{m=1}^Mw_0^{r_m}\prod_{n=1}^N\prod_{k=1}^K w_{k}^{r_m,l_k^n}\right)\\
&\qquad\qquad\qquad\times f((\theta^{r_m},x^{r_m,l^1_1}_1,\dots,x^{r_m,l^N_1}_1,\dots,x^{r_m,l^1_K}_K,\dots,x^{r_m,l^N_K}_K)_{m=1}^M)\Bigg]
=:g(w_0^{1:M},w_{1:K}^{1:M,1:N}),
\end{align*}
which holds for all  $w_0^{1:M}$ in  $\cal{P}([M])$ and $w_{1}^{1,1:N},\dots,w_{K}^{1,1:N}$, $\dots$, $w_{1}^{M,1:N},\dots,w_{K}^{M,1:N}$ in  $\cal{P}([N])$. The same type of argument as before then shows that $g$ is zero on the positive orthant $\r^{M(1+NK)}_{>0}$:  for any $z_0^{1:M}$ in $\r^M$ and  $z_{1:K}^{1:M,1:N}$ in $\r^{MNK}_{>0}$, set 
$$w_0^r:= \frac{z_0^{r}(s_1^r\dots s_K^r)^N}{s_0}\enskip\text{and}\enskip w_k^{r,1:N}:=\frac{z_k^{r,1:N}}{s_k^r},\quad\text{where} \quad  s_0:=\sum_{r=1}^Mz_0^r(s_1^r\dots s_K^r)^N\enskip\text{and}\enskip s_k^r:=\sum_{l=1}^Nz_k^{r,l},$$
in which case
\begin{align*}
\left(\prod_{m=1}^Mz_0^{r_m}\prod_{n=1}^N\prod_{k=1}^K z_{k}^{r_m,l_k^n}\right)=\left(\prod_{m=1}^Mz_0^{r_m}(s_1^{r_m}\dots s_K^{r_m})^N\prod_{n=1}^N\prod_{k=1}^K w_{k}^{r_m,l_k^n}\right)=s_0^M\left(\prod_{m=1}^Mw_0^{r_m}\prod_{n=1}^N\prod_{k=1}^K w_{k}^{r_m,l_k^n}\right);
\end{align*}
whence $g(z_0^{1:M},z_{1:K}^{1:M,1:N})=s_0^Mg(w_0^{1:M},w_{1:K}^{1:M,1:N})=0$. Just as before, it follows that $g$ must be zero everywhere. This time around, $g$ features monomials of the type 
\begin{equation}\label{eq:monppf}\prod_{r=1}^M[w_0^r]^{d_0^r}\prod_{l=1}^N\prod_{k=1}^K[w_k^{r,l}]^{d_k^{r,l}}\end{equation}
with
\begin{equation}\label{eq:degppf}\sum_{r=1}^Md_0^r=M,\quad\sum_{r=1}^M\sum_{l=1}^Nd_k^{r,l}=MN \enskip\forall k\in[K].\end{equation}
The coefficient $c$ of~\eqref{eq:monppf} in $g$ is
\begin{equation}\label{eq:mgeau9gnusadanga}c=\sum_{r_{1:M},l^{1:M,1:N}_{1:K}}f((\theta^{r_m},x^{r_m,l^1_1}_1,\dots,x^{r_m,l^N_1}_1,\dots,x^{r_m,l^1_K}_K,\dots,x^{r_m,l^N_K}_K)_{m=1}^M),\end{equation}
where the sum is taken over all indices $r_{1:M}$ in $[M]^M$ and $l^{1:N}_{1:K}$ in $[N]^{NK}$ satisfying
\begin{align*}
\mmag{\{m\in[M]:r_m=r\}}=d_0^r,\quad \mmag{\{(m,n)\in[M]\times [N]:(r_m,l_k^n)=(r,l)\}}=d_k^{r,l},
\end{align*}
for all $r$ in $[M]$, $l$ in $[N]$, and $k$ in $[K]$. The symmetries in $f$ then imply that all terms in~\eqref{eq:mgeau9gnusadanga}'s right-hand side are the same, so we find that 
$$f((\theta^{r_m},x^{r_m,l^1_1}_1,\dots,x^{r_m,l^N_1}_1,\dots,x^{r_m,l^1_K}_K,\dots,x^{r_m,l^N_K}_K)_{m=1}^M)=0.$$
Because $r_m:=m$ for all $m$ in $[M]$ and $l_k^n:=n$ for all $n$ in $[N]$ and $k$ in $[K]$ satisfy~\eqref{eq:degppf}, the above implies \eqref{eq:fwmua9ifnwuanaf}.
\end{proof}
The rest of the argument entails an  application of the Cauchy-Schwartz inequality analogous to that in Appendix~\ref{app:mvu}. (Only that, this time, we `symmetrize' a generic estimator $f$ using
\begin{align*}\frac{1}{M!(N!)^{MK}}\sum_{\upsilon,\tau_{1:K}^{1:M}}f((\theta^{\upsilon(m)},x_1^{\upsilon(m),\tau_1^m(1)},\dots,x_1^{\upsilon(m),\tau_1^m(N)},\dots,x_K^{\upsilon(m),\tau_K^m(1)},\dots,x_K^{\upsilon(m),\tau_K^m(N)})_{m=1}^M),
\end{align*}
where the sum is taken over all permutations $\upsilon$ of $[M]$ and $\tau_1^1,\dots,\tau_K^1,\dots,\tau_1^M,\dots,\tau_K^M$  of $[N]$.)

\section{Proof of Corollary~\ref{cor:pmh}}\label{app:pmh}

This corollary follows from Theorem~10 and Remark 12 in \cite{Andrieu2016} and the fact that, for all $\theta$ in $\Theta$ and $N>0$, $\pi^N_\times(\theta)$ is bounded above by  $\pi^N(\theta)$ in the convex order:
\begin{equation}\label{eq:convexorder}
\Ebb{f(\pi^N_\times(\theta))}\leq \Ebb{f(\pi^N(\theta))}
\end{equation}
for all convex functions $f:\r\to\r$. To argue the above, fix any such $f$, $\theta$, and $N$, and note that
\begin{align*}\pi^N_\times(\theta)&=\frac{1}{N^K}\sum_{\bm{n}\in[N]^K}w(\theta,X^{n_1}_1,X^{n_2}_2,\dots,X^{n_K}_K)\\
&=\frac{1}{N^{K-1}}\sum_{\bm{m}\in[N]^{K-1}}\underbrace{\frac{1}{N}\sum_{n=1}^Nw(\theta,X^{n}_1,X^{(n+m_1\text{ mod }N)+1}_2,\dots,X^{(n+m_{K-1}\text{ mod }N)+1}_K)}_{=:Y^N_{\bm{m}}}.
\end{align*}
Because, for any $\bm{m}$ in $[N]^{K-1}$,  $(X^{m_1}_2,\dots,X^{m_{K-1}}_K)$ has law $\pi(dx_2,\dots,dx_K)$ and is independent of $(X_1^1,\dots,X_1^N)$, $(Y^N_{\bm{m}})_{\bm{m}\in[N]^{K-1}}$ is a collection of identically distributed random variables.
Moreover, by definition $\pi^N(\theta)=Y^N_{(N-1,\dots, N-1)}$, so $(Y^N_{\bm{m}})_{\bm{m}\in[N]^{K-1}}$ all share the same distribution as $\pi^N(\theta)$ and~\eqref{eq:convexorder} follows from Jensen's inequality:
\begin{align*}
\Ebb{f(\pi^N_\times(\theta))}&=\Ebb{f\left(\frac{1}{N^{K-1}}\sum_{\bm{n'}\in[N]^{K-1}}Y^N_{\bm{m}}\right)}\leq \frac{1}{N^{K-1}}\sum_{\bm{m}\in[N]^{K-1}}\Ebb{f(Y^N_{\bm{m}})}\\
&=\frac{1}{N^{K-1}}\sum_{\bm{m}\in[N]^{K-1}}\Ebb{f(\pi^N(\theta))}=\Ebb{f(\pi^N(\theta))}.
\end{align*}
\section{Estimators for targets that are mixtures of product-form distributions}\label{sec:mop}Suppose that our target $\mu$ is not a product-form distribution but a mixture of several:
\begin{equation}\label{eq:moptarget}\mu:=\sum_{i=1}^I\theta_i\mu^i,\end{equation}
where the mixture weights $\theta_1,\ldots,\theta_I>0$ satisfy $\sum_{i=1}^I\theta_i=1$ and, for each $i$ in $[I]$, $\mu^i$ is the product $\prod_{k=1}^{K_i}\mu^i_{k}$  of distributions $\mu_1^i,\ldots,\mu_{K_i}^i$ respectively defined on
$$(S_{k}^i,\cal{S}_{k}^i):=\left(\prod_{k'\in \cal{K}_k^i}S_{k'},\prod_{k\in \cal{K}_1^i}\cal{S}_{k'}\right)\quad\forall k=1,\dots,K_i,$$%
for some given partitions $\{\cal{K}_1^1,\dots,\cal{K}_{K_1}^1\},\dots,\{\cal{K}_1^I,\dots,$ $\cal{K}_{K_I}^I\}$  of $[K]$. Notice that, to further broaden the applicability of the ensuing estimators, we allow for mixture components that are products over some but not all dimensions (in the case where every mixture component does factorize fully, $K_i=K$ and $\cal{K}_k^i=\{k\}$ for each $i$ in $[I]$ and $k$ in $[K]$). 
Fix a square $\mu$-integrable test function $\varphi$ and suppose we wish to compute $\mu(\varphi)$. It is well-known~\cite{Oh1993,Hesterberg1995} that the basic Monte Carlo estimator, 
$$\mu^N(\varphi):=\frac{1}{N}\sum_{n=1}^N\varphi(X^n)$$ 
where $X^1,\dots,X^N$ denote i.i.d.\ samples drawn from $\mu$, is outperformed by its stratified variant:
\begin{equation}\label{eq:stratified}
\mu^N_{s}(\varphi):=\sum_{i=1}^I\frac{\theta_i}{N_i}\sum_{n=1}^{N_i}\varphi(X^{i,n})=:\sum_{i=1}^I\theta_i\mu^{i,N_i}(\varphi),\end{equation}
where $(X^{1,1},\dots, X^{1,N_1}),\dots,(X^{I,1},\dots, X^{I,N_I})$ are $I$ independent sequences of i.i.d.\ samples respectively drawn from $\mu^1,\dots,\mu^I$ and  $N_1=\theta_1 N,\dots,N_I=\theta_I N$ for some $N>0$ (out of convenience, we are assuming that these are integers). In particular, $\mu^N_{s}(\varphi)$ is a consistent, unbiased, and asymptotically normal estimator for $\mu(\varphi)$ whose variance is bounded above by that of $\mu^N(\varphi)$:
\begin{align*}\text{Var}(\mu^N_{s}(\varphi))&=\sum_{i=1}^I\theta_i^2\text{Var}(\mu^{i,N_{i}}(\varphi))=\sum_{i=1}^I\theta_i^2\frac{\mu^i([\varphi-\mu^i(\varphi)]^2)}{N_i}=\frac{\sum_{i=1}^I\theta_i\mu^i([\varphi-\mu^i(\varphi)]^2)}{N}\\
&=\frac{\sum_{i=1}^I\theta_i\mu^i(\varphi^2)-\sum_{i=1}^I\theta_i\mu^i(\varphi)^2}{N}\leq \frac{\mu(\varphi^2)-\mu(\varphi)^2}{N}
=\frac{\sigma^2(\varphi)}{N}=\text{Var}(\mu^N(\varphi))\quad\forall N>0.\end{align*}
Given our assumption that the distributions in the mixture are product-form, we easily improve on $\mu^N_s(\varphi)$ using the product-form analogues of $\mu^{1,N_1}(\varphi),\dots, \mu^{I,N_I}(\varphi)$: 
\begin{equation}
\label{eq:mopest}\mu^N_{s,\times}(\varphi):=\sum_{i=1}^I\frac{\theta_i}{N_i^{K_i}}\sum_{\bm{n^i}\in[N_i]^{K_i}}\varphi(X^{i,\bm{n^i}})=:\sum_{i=1}^I\theta_i\mu^{i,N_i}_\times(\varphi),\end{equation}
is more statistically efficient than $\mu^N_s(\varphi)$. In particular:
\begin{corollary}\label{cor:MOP}If $\varphi$ is $\mu$-integrable, then $\mu^N_{s,\times}(\varphi)$ in~\eqref{eq:mopest} is an unbiased estimator for $\mu(\varphi)$. If $\varphi$ belongs to $L^2_\mu$ and $N_1,\dots,N_I$ are integers proportional to $N$ (i.e.\ $N_1=\alpha_1 N,\dots,N_I=\alpha_I N$ for some $\alpha_1,\dots,\alpha_I>0$), then $\mu^N_{s,\times}(\varphi)$  is strongly consistent and asymptotically normal with variance
\begin{align*}
  &\text{Var}(\mu^N_{s,\times}(\varphi))=\sum_{i=1}^I\theta_i^2\text{Var}(\mu^{i,N_i}_\times(\varphi))=\sum_{i=1}^I\theta_i^2\left(\sum_{\emptyset\neq A\subseteq [K_i]}\frac{\sum_{B\subseteq A}(-1)^{\mmag{A}-\mmag{B}}[\sigma_{A,B}^i(\mu_{A^c}^i(\varphi))]^2}{\alpha_i^{\mmag{A}}N^{\mmag{A}}}\right)
\end{align*}
for all $N>0$, and asymptotic variance 
$$\sigma^2_{s,\times}(\varphi)=\sum_{i=1}^I\frac{[\theta_i\sigma^i_\times(\varphi)]^2}{\alpha_i},$$
where, for all $\psi$ in $ L_{\mu^i_A}^2$, $B\subseteq A\subseteq[K_i]$, and $i$ in $[I]$,
\begin{align*}[\sigma_{A,B}^i(\psi)]^2:&=\mu_A^i([\mu_{A\backslash B}^i(\psi)-\mu_A^i(\psi)]^2),\\
[\sigma^i_\times(\varphi)]^2:&=\sum_{k=1}^{K_i}\mu^i_k([\mu^i_{[K_i]\backslash\{k\}}(\varphi)-\mu^i(\varphi)]^2).\end{align*}
%
If, in particular, $\alpha_1=\theta_1,\dots,\alpha_I=\theta_I$, then the variances of $\mu^N_{s,\times}$ are bounded above by those of  $\mu_s^N$:  
\begin{align*}
\text{Var}(\mu^N_{s,\times}(\varphi))&=\sum_{i=1}^I\theta_i^2\text{Var}(\mu^{i,N_i}_\times(\varphi))\leq \sum_{i=1}^I\theta_i^2\text{Var}(\mu^{i,N_i}(\varphi))=\text{Var}(\mu^N_{s}(\varphi))\quad\forall N>0,\\
\sigma^2_{s,\times}(\varphi)&=\sum_{i=1}^I\theta_i[\sigma^i_\times(\varphi)]^2\leq\sum_{i=1}^I\theta_i^2\mu^i([\varphi-\mu^i(\varphi)]^2)=\lim_{N\to\infty} N\text{Var}(\mu^N_{s}(\varphi)).
\end{align*}
\end{corollary}
\begin{proof}Follows from Theorem~\ref{thrm:classicalK}, Corollary~\ref{cor:varbound}, and the continuous mapping theorem.
\end{proof}
Employing a Lagrange multiplier as in~\cite[p.153]{Asmussen2007}, we find that  $N_i^*\propto \theta_i\sigma^i_\times(\varphi)$ for all $i$ in $[I]$ 
achieve the smallest possible asymptotic variance ($\sum_{i=1}^I\theta_i\sigma^i_\times(\varphi)$) for fixed $N=\sum_{i=1}^I N_i$. 
Of course, $\sigma_\times^1(\varphi)$, $\dots$, $\sigma_\times^I(\varphi)$ are unknown in practice and must be estimated, which would lead to an adaptive scheme similar to those for $\mu^N_s(\varphi)$ in~\cite{Raghavan1998,Etore2010,Douc2007}. 
\end{document}